\documentclass[a4paper,11pt]{article}
\usepackage[ngerman, english]{babel}
\usepackage[T1]{fontenc}
\usepackage[utf8]{inputenc}
\usepackage{amsmath}
\usepackage{amssymb}
\usepackage{amsthm}
\usepackage{graphicx}
\usepackage{here}
\usepackage{color}
\usepackage{xcolor}
\usepackage{enumerate}
\usepackage{lmodern}
\usepackage{fancyvrb}
\usepackage[plainpages=false]{hyperref}
\usepackage{caption}
\usepackage{subfigure}
\newcommand{\ca}{\mathcal}

\newtheorem{defi}{Definition}[section]
\newtheorem{theo}[defi]{Theorem}
\newtheorem{lem}[defi]{Lemma}
\newtheorem{con}[defi]{Conjecture}
\newtheorem{obs}[defi]{Observation}
\newtheorem{cor}[defi]{Corollary}
\newtheorem{prob}[defi]{Problem}
\theoremstyle{remark}

\title{Pairwise disjoint perfect matchings in $r$-edge-connected $r$-regular graphs}

\author{Yulai Ma$^1$\thanks{Supported by Sino-German (CSC-DAAD) Postdoc Scholarship Program 2021 (57575640)}, Davide Mattiolo$^2$\thanks{Supported by a Postdoctoral Fellowship of the Research Foundation Flanders (FWO), project number 1268323N}, Eckhard Steffen$^1$, Isaak H.~Wolf$^1$ \thanks{Funded by the Deutsche Forschungsgemeinschaft (DFG, German Research Foundation) – 445863039} \\
\footnotesize
$^1$ Department of Mathematics, Paderborn University, Warburger Str.\ 100, 33098 Paderborn,
Germany.
\\
\footnotesize
$^2$ Department of Computer Science, KU Leuven Kulak, 8500 Kortrijk, Belgium.
\\ 
\footnotesize yulai.ma@upb.de, davide.mattiolo@kuleuven.be, es@upb.de, isaak.wolf@upb.de}

\date{}

\begin{document}

\maketitle

\begin{abstract}
Thomassen [Problem 1 in Factorizing regular graphs, \emph{J. Combin. Theory Ser. B}, 141 (2020), 343-351] asked whether 
every $r$-edge-connected $r$-regular graph of even order has $r-2$
pairwise disjoint perfect matchings. We show that this is not the case
if $r \equiv 2 \text{ mod } 4$. Together with a recent result of Mattiolo and Steffen [Highly edge-connected regular graphs without large factorizable subgraphs, \emph{J. Graph Theory}, 99 (2022), 107-116]
this solves Thomassen's problem for all even $r$.

It turns out that our methods are limited to the even case of Thomassen's problem. We then prove some equivalences of statements on pairwise disjoint perfect matchings in highly edge-connected regular graphs, where the perfect matchings contain or avoid fixed sets of edges. 

Based on these results we relate statements on pairwise disjoint perfect matchings
of 5-edge-connected 5-regular graphs to well-known conjectures for cubic graphs, such as the Fan-Raspaud Conjecture, the Berge-Fulkerson Conjecture and the $5$-Cycle Double Cover Conjecture.
\end{abstract}

{\bf Keywords:} perfect matchings, regular graphs, factors, $r$-graphs, edge-colorings, class $2$ graphs

\section{Introduction and motivation}
We consider finite graphs that may have parallel edges but no loops. A graph without parallel edges is called \emph{simple}.
Let $r \geq 0$ be an integer. A graph $G$ is $r$-\emph{regular} if every vertex has degree $r$, where the \emph{degree} $d_G(v)$ of a vertex $v$ is the number of edges that are incident with $v$. An $r$-\emph{graph} is an $r$-regular graph $G$, where
every odd set $X \subseteq V(G)$ is connected by at least $r$ edges to its complement
$V(G) \setminus X$.
For $k \in \{1, \dots, r\}$, a $k$-\emph{factor} of $G$ is a spanning $k$-regular subgraph of $G$. The edge set of a $1$-factor is called a \emph{perfect matching} of $G$. Moreover, for $k\ge 2$, a $k$-\emph{PDPM} of an $r$-regular graph $G$ is a set of $k$ pairwise disjoint perfect matchings of $G$.

An $r$-regular graph is \emph{class 1} if it has an $r$-PDPM. 
Otherwise, it is \emph{class 2}.  
On one side, dense simple $r$-regular graphs are class $1$ \cite{csaba2016proof}. 
On the other side, for every $r \ge 3$ there are $(r-2)$-edge-connected $r$-regular graphs of even order without a perfect matching, see for example \cite{factors_and_factorizations_book} p.~47 ff.

Every $r$-graph has a perfect matching \cite{seymour1979multi} and we call an
$r$-graph \emph{poorly matchable} if any two perfect matchings intersect. 
Every class $2$ $3$-graph is poorly matchable. A natural question is
what is the maximum number $t$ such that every $r$-graph has a $t$-PDPM? 
For $r$-graphs, the answer is given by
Rizzi \cite{rizzi1999indecomposable}, who showed that for every $r \geq 3$, there are 
poorly matchable $r$-graphs. However, all graphs he constructed
have a 4-edge-cut. It seems that the situation changes for
highly edge-connected $r$-graphs.

Let $m(r)$ be the maximum number $t$ such that every $r$-edge-connected
$r$-graph has $t$ pairwise disjoint perfect matchings. This gives rise to the following 
problem.

\begin{prob} \label{mainProb}
	Determine $m(r)$ for all $r \geq 2$. 
\end{prob}

Clearly,  $m(r) \geq 1$, $m(2)=2$, and $m(3) = 1$. Furthermore, $m(4)=1$ by the result of Rizzi \cite{rizzi1999indecomposable}.  Thomassen \cite{THOMASSEN2020343} conjectured that 
there exists a natural number $r_0$ such that $m(r)\geq 2$ for every $r \geq r_0$.
In other words, he conjectured that for all $r \geq r_0$, there are no poorly matchable
$r$-edge-connected $r$-graphs. Thus, $r_0 \geq 5$. 

Class 1 graphs are of no interest for the study of $m(r)$. 
They have $r$ pairwise disjoint perfect matchings.
For $r \geq 3$, class $2$ $r$-graphs have at most an $(r-2)$-PDPM.
Thus, 
$m(r) \leq r-2$, if there is an $r$-edge-connected class $2$ $r$-graph.
Note that $r$-edge-connected class $2$ $r$-graphs are known for
all $r \geq 3$ and $r \not= 5$. 
The case $r=5$ 
is briefly addressed at the end of this section and in Section \ref{sec:disjoint_pm_5_reg_graphs}. 

Thomassen (Problem 1 of \cite{THOMASSEN2020343}) proposed the following question
for the value of $m(r)$. 

\begin{prob}[Thomassen \cite{THOMASSEN2020343}] \label{Thom-r-2}
	For all $r \geq 3$, is it true that $m(r) = r-2$?
\end{prob}

Problem \ref{Thom-r-2} is solved for $r\equiv 0 \mod4$ in \cite{Mattiolo2022HighlyER}.
It is shown that $m(r) \leq r-3$ in this case. 
In Section \ref{Sec: m(4t+2)} we prove the following statement. 

\begin{theo}\label{4k+1con}
	If $r\equiv 2 \mod4$, then $m(r) \leq r-3$. 
\end{theo}

Together with the results of \cite{Mattiolo2022HighlyER, rizzi1999indecomposable} 
we obtain the following corollary. 

\begin{cor} \label{cor: disprove Thomassen even}
	If $r > 2$ is even, then $m(r) \leq r-3$.
\end{cor}

The graphs that prove Corollary \ref{cor: disprove Thomassen even} have a 2-vertex-cut. 
It is easy to see that for odd $r$, an $r$-edge-connected $r$-graph
is $3$-vertex-connected. This shows that our methods are limited to the case when 
$r$ is even. 

Thus, the main motivation for Section \ref{sec:Equivalences} is the study of
Problems \ref{mainProb} and \ref{Thom-r-2} for odd $r$. We prove that 
every $r$-edge-connected $r$-graph has 
$k \in \{2, \dots, r-2\}$ pairwise
disjoint perfect matchings if and only if 
every $r$-edge-connected $r$-graph has $k$ pairwise disjoint perfect matchings that contain (or that avoid) a fixed edge. 
For odd $r$,
we prove the stronger statement that 
every $r$-edge-connected $r$-graph has an $(r-2)$-PDPM if and only if 
for every $r$-edge-connected $r$-graph and every $\lfloor \frac{r}{2} \rfloor$ 
adjacent edges, there is 
an $(r-2)$-PDPM of $G$ containing all $\lfloor \frac{r}{2} \rfloor$ edges. 

In Section \ref{sec:disjoint_pm_5_reg_graphs} these results are used to
prove tight connections between (possible) answers to
Problem \ref{mainProb} for $r=5$ and some well-known conjectures 
on cubic graphs. In particular, we prove the following statement. 

\begin{theo}\label{FR-m5}
	If $m(5) \geq 2$, then the Fan-Raspaud Conjecture holds. Moreover, if $m(5) = 5$, then both the $5$-Cycle Double Cover Conjecture and the Berge-Fulkerson Conjecture hold.
\end{theo}

The condition $m(5)=5$ of the second statement of Theorem \ref{FR-m5} 
seems to be surprising since it is equivalent to saying that every $5$-edge-connected $5$-regular graph is class $1$. So far, we have not succeeded in constructing 
$5$-edge-connected $5$-regular class $2$ graphs. Also,
intensive literature research and computer-assisted searches in
graph databases did not lead to the desired success. 
We conclude this introduction with the following problem, which surprisingly seems to be unsolved. 

\begin{prob}\label{prob:5_reg_class_2}
	Is there any $5$-edge-connected $5$-regular class $2$ graph?
\end{prob}

For planar graphs, the answer to the above question is ``no''. 
Guenin \cite{guenin2003packing} proved that all planar 5-graphs are class $1$. 
Indeed, it is conjectured by Seymour \cite{seymour1979multi} that every planar $r$-graph is class $1$.
So far, this conjecture is proved to be true for all $r \leq 8$, see \cite{planar8-graphsclass1} also for further references.

\section{Basic definitions and results}

In this section we introduce some definitions and tools. For undefined notation and terminology the readers are referred to \cite{Bondy2008}.

If $u,v$ are two vertices of a graph $G$, we denote by $\mu_G(u,v)$ the number of parallel edges connecting $u$ and $v$. 
In case $\mu_G(u,v)=1$ we also say that $e=uv$ is a simple edge.
The 
\emph{underlying graph} of $G$ is the simple graph 
$H$ with $V(H)=V(G)$ and $\mu_H(u,v)=1$ if and only if $\mu_G(u,v)  \geq 1$.  
For any two disjoint sets $U,  W\subseteq V(G)$, denote by 
$[U, W]_G$  the set of  edges with exactly one endpoint in each of  $U$ and $W$. For convenience, we simply write $[u,W]_G$ for $[U, W]_G$ if $U= \{u\}$, and just write $\partial_G(U)$ if $W = V(G)\setminus U$.
Moreover, $\partial_G(U)$ is called an {\em edge-cut} of $G$.
The index $G$ is sometimes omitted if there is no ambiguity. 
The graph induced by $U$ is denoted by $G[U]$.  Similarly, we use $G-U$ instead of  $G[V(G)\setminus U]$, and use $G-u$ instead of $G-\{u\}$ if $U= \{u\}$. A {\em $k$-circuit} is a circuit of length $k$ and 
{\em a cycle} is a union of pairwise edge-disjoint circuits.
A set of $k$ vertices, whose removal increases the number of components of a graph, is called a {\em $k$-vertex-cut}.
A graph $G$ on at least $k+1$ vertices is called {\em $k$-connected} if $G-X$ is connected for every $X\subseteq V(G)$ with $|X|\leq k-1$.
Similarly, a graph $G$ is called {\em $k$-edge-connected} if for every non-empty $X \subset V(G), |\partial_G(X)| \ge k$. Moreover, a graph $G$ is called {\em cyclically $k$-edge-connected} if $G-S$ has at most one component containing a circuit for any edge-cut $S \subseteq E(G)$ with $|S|\leq k-1$.

A \emph{multiset} $\ca M$ consists of objects with possible repetitions. We denote by $|\ca M|$ the number of objects in $\ca M$. For a positive integer $k$, we define $k\ca M$ to be the multiset consisting of $k$ copies of each element of $\ca M$.
Let $G$ be a graph and $N$ a multiset of edges of the complete graph on $V(G)$. The graph $G+N$ is obtained by adding a copy of all edges of $N$ to $G$. This operation might generate parallel edges. More precisely, if $N$ contains exactly $t$ edges connecting the vertices $u$ and $v$ of $G$, then $\mu_{G+N}(u,v)=\mu_G(u,v)+t$.

We will frequently use the following simple fact.

\begin{obs}\label{odd-insect}
	Let $G$ be a graph with a perfect matching $M$. For any subset $X\subseteq V(G)$, if $|X|$ is odd, then $|\partial_G (X)\cap M|$ is odd.
\end{obs}

Many results in the following sections rely on the properties of the Petersen graph, which will be denoted by $P$ throughout this paper. We need to fix a drawing of $P$ and consider its six perfect matchings.  Let $v_1\dots v_5v_1$ and $u_1u_3u_5u_2u_4u_1$ be the two disjoint $5$-circuits of $P$ such that $u_iv_i\in E(P)$ for each $i\in\{1,\ldots,5\} $. Set $M_0=\{u_iv_i: i \in \{1, \dots,5\}\}$. Clearly, $M_0$ is a perfect matching of the Petersen graph. For each $i\in\{1,\dots,5\}$, let $M_i$ be the only other perfect matching containing $u_iv_i$, see Figure \ref{fig:pm_P}.

\begin{figure}[htbp]
\centering
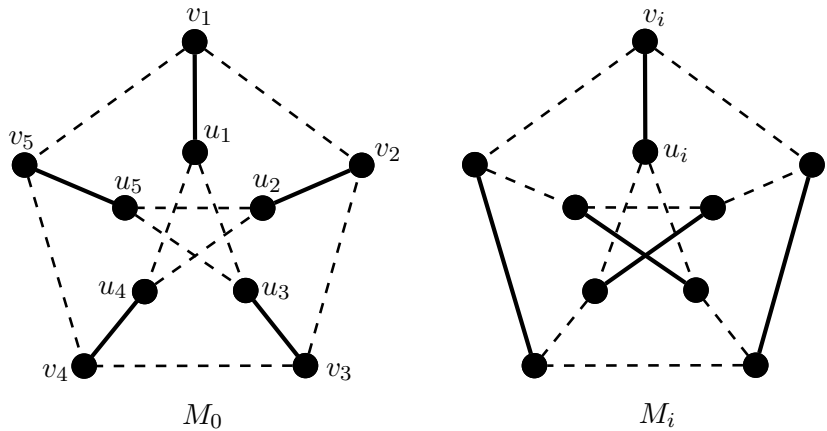
\caption{The $6$ perfect matchings of the Petersen graph.}
\label{fig:pm_P}
\end{figure}

Let $P^{\ca M}=P+\sum_{j=1}^k N_j$ be the graph obtained from $P$ and $\ca M$, where $\ca M = \{N_1,\dots,N_k\}$ is a multiset of perfect matchings of $P$.
For a perfect matching $N$ of 
$P^{\ca M}$,  if it is a copy of $M_j$ of $P$, then we say that $N$ is of type $j$ in $P^{\ca M}$. Let $\ca M'$ be a multiset of perfect matchings of $P^{\ca M}$.
We denote by $\ca M'_P $ the multiset of perfect matchings of $\ca M'$ interpreted as perfect matchings of $P$. That is, for each $j\in\{0,\dots,5\}$, $\ca M'_P$ contains exactly $k$ copies of $M_j$ if and only if $\ca M'$ contains exactly $k$ perfect matchings of type $j$.

\begin{lem}[\cite{Mattiolo2022HighlyER}]\label{lem:P+matchings}
Let $\ca M$ be a multiset of $k$ perfect matchings of $P$. If $\ca M' = \{M_1',\dots,M_{k+1}'\}$ is a $(k+1)$-PDPM of $P^{\ca M}$, then $ \ca M \subseteq \ca M'_{P}$.
	\end{lem}

We 
also need the following lemma, whose proof is basically the same as that of Lemma 2.4 in \cite{Mattiolo2022HighlyER} but its statement is more general.
To keep the paper self-contained, the proof is presented.

\begin{lem}\label{lem:P+matchings_Connectivity}
Let $\ca M$ be a multiset of $k$ perfect matchings of $P$. If for every $u,v\in V(P^{\ca M})$,  $\mu_{P^{\ca M}}(u,v) \le \lfloor \frac{k+3}{2} \rfloor $, then $P^{\ca M}$ is a $(k+3)$-edge-connected $(k+3)$-regular class $2$ graph.
\end{lem}
	
\begin{proof}
By construction, $P^{\ca M}$ is $(k+3)$-regular. Moreover, it is class $2$, see Theorem 3.1 of \cite{Grunewald_Steffen_1999}.
Let $X\subset  V(P^{\ca M})$ be the subset with $|\partial_{P^{\ca M}} (X)|$ minimum. This implies that 
$P^{\ca M}[X]$ is connected. If $|X|$ is odd, then  by Observation \ref{odd-insect}, every perfect matching of $\ca M$ intersects $\partial_{P^{\ca M}} (X)$. Hence, $|\partial_{P^{\ca M}} (X)|\ge 3 + k$. If $|X|$ is even, then it suffices to consider the cases $|X|\in\{2,4\}$. Since $P$ has girth $5$, the subgraph $P[X]$ is a path on either $2$ or $4$ vertices or
it is isomorphic to $K_{1,3}$. Since $\mu_{P^{\ca M}}(u,v) \le \lfloor \frac{k+3}{2} \rfloor $ for every $u,v\in V(P^{\ca M})$ it follows that $\partial_{P^{\ca M}}(X)$ contains at least $ \lceil \frac{k+3}{2} \rceil$ edges for each vertex of degree 1
in $P[X]$. Consequently, $|\partial_{P^{\ca M}}(X)|\ge k+3 $.
	\end{proof}

Meredith \cite{Meredith} constructed $r$-edge-connected class $2$  $r$-graphs for every $r\ge 8$. One can observe that 
$P+M_0+M_1+M_2$ and $P+M_0+M_1+M_2+M_3$ are respectively a  $6$-edge-connected $6$-regular graph and a $7$-edge-connected $7$-regular graph. Such graphs are class $2$ by  \cite{Grunewald_Steffen_1999} (Theorem 3.1). However, we cannot easily construct a $5$-edge-connected $5$-regular graph in the same way. Indeed, adding 
two perfect matchings to $P$ generates a $4$-edge-cut. It is also 
well known that for each $r \in \{3,4\}$ there are $r$-edge-connected class $2$ $r$-graphs.
Surprisingly, it seems that so far no $5$-edge-connected $5$-regular class $2$ graph is known, see Problem \ref{prob:5_reg_class_2}.

\section{Proof of Theorem \ref{4k+1con}}\label{Sec: m(4t+2)}

In this section we construct a $(4k+2)$-edge-connected $(4k+2)$-graph $G_k$ without a $4k$-PDPM for each integer $k \geq 1$. As in \cite{Mattiolo2022HighlyER}, we first construct a graph $P_k$ by adding perfect matchings to the Petersen graph and a graph $Q_k$ by using two copies of $P_k$. Then, we construct a graph $S_k$ and "replace" some edges of $S_k$ by copies of $Q_k$ to obtain the graph $G_k$ with the desired properties.

\subsection{The graphs $P_k$ and $Q_k$}

For each $k \geq 1$, let $$P_k=P+k(M_0+M_1+M_2)+(k-1)M_5,$$ as shown in Figure \ref{Fig:P_2}.

\begin{figure}[htbp]
	\centering
	\includegraphics[scale=0.6]{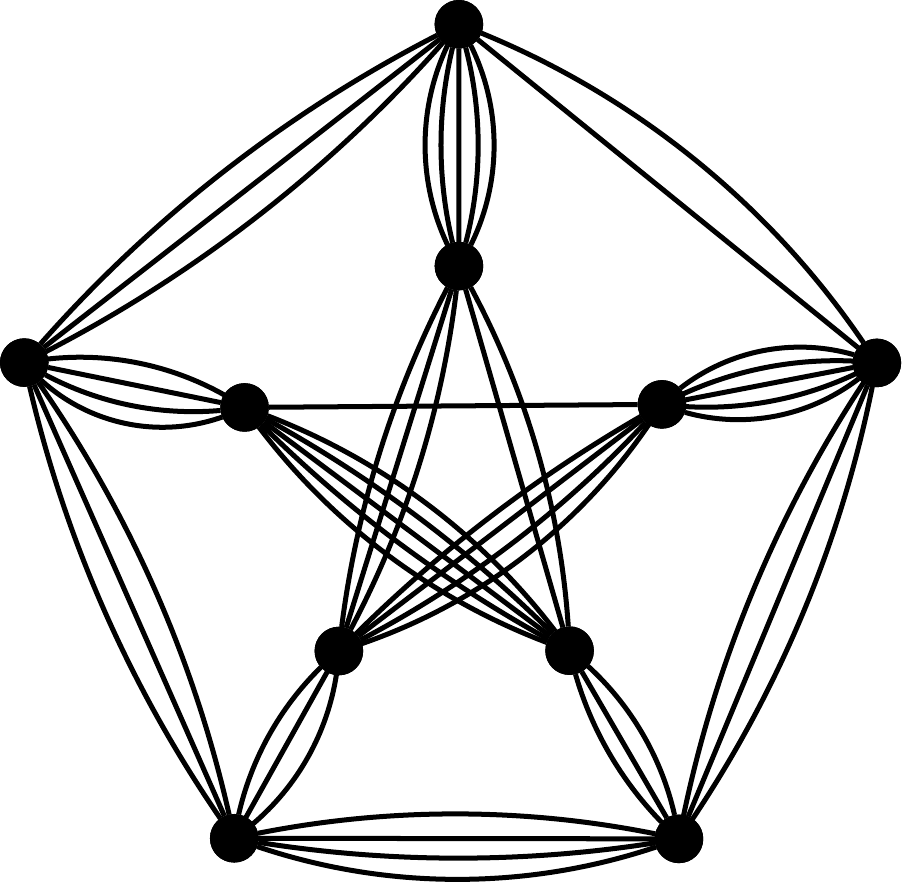}
	\caption{The graph $P_2$.}\label{Fig:P_2}
\end{figure}

Let $P_k^1$ and $P_k^2$ be two distinct copies of $P_k$. For each $w\in V(P_k)$, the vertex of $P_k^1$ ($P_k^2$, respectively) that corresponds to $w$ is denoted by $w^1$  ($w^2$, respectively). Now, we obtain the graph $Q_k$ from $P_k^1$ and $P_k^2$ by removing the $2k+1$ parallel edges connecting $u_1^i$ and $v_1^i$ from $P_k^i$, for each $i \in \{1,2\}$, and identifying $u_1^1$ and $u_1^2$ to a new vertex, denoted by $u_{Q_k}$. Note that the degree of $u_{Q_k}$ in $Q_k$ is $4k+2$. For a graph $G$ containing $Q_k$ as a subgraph, let $E_k^i=[v_1^i, V(G) \setminus V(Q_k)]_G$ for each $i\in \{1,2\}$. 
The subgraph $Q_2$
and the edge-sets $E^1_2$ and $E^2_2$ are shown in Figure~\ref{Fig:Q_2}. 

\begin{figure}[htbp]
\centering
\begingroup%
  \makeatletter%
  \providecommand\color[2][]{%
    \errmessage{(Inkscape) Color is used for the text in Inkscape, but the package 'color.sty' is not loaded}%
    \renewcommand\color[2][]{}%
  }%
  \providecommand\transparent[1]{%
    \errmessage{(Inkscape) Transparency is used (non-zero) for the text in Inkscape, but the package 'transparent.sty' is not loaded}%
    \renewcommand\transparent[1]{}%
  }%
  \providecommand\rotatebox[2]{#2}%
  \newcommand*\fsize{\dimexpr\f@size pt\relax}%
  \newcommand*\lineheight[1]{\fontsize{\fsize}{#1\fsize}\selectfont}%
  \ifx\svgwidth\undefined%
    \setlength{\unitlength}{377.37342172bp}%
    \ifx\svgscale\undefined%
      \relax%
    \else%
      \setlength{\unitlength}{\unitlength * \real{\svgscale}}%
    \fi%
  \else%
    \setlength{\unitlength}{\svgwidth}%
  \fi%
  \global\let\svgwidth\undefined%
  \global\let\svgscale\undefined%
  \makeatother%
  \begin{picture}(1,0.36383494)%
    \lineheight{1}%
    \setlength\tabcolsep{0pt}%
    \put(0,0){\includegraphics[width=\unitlength,page=1]{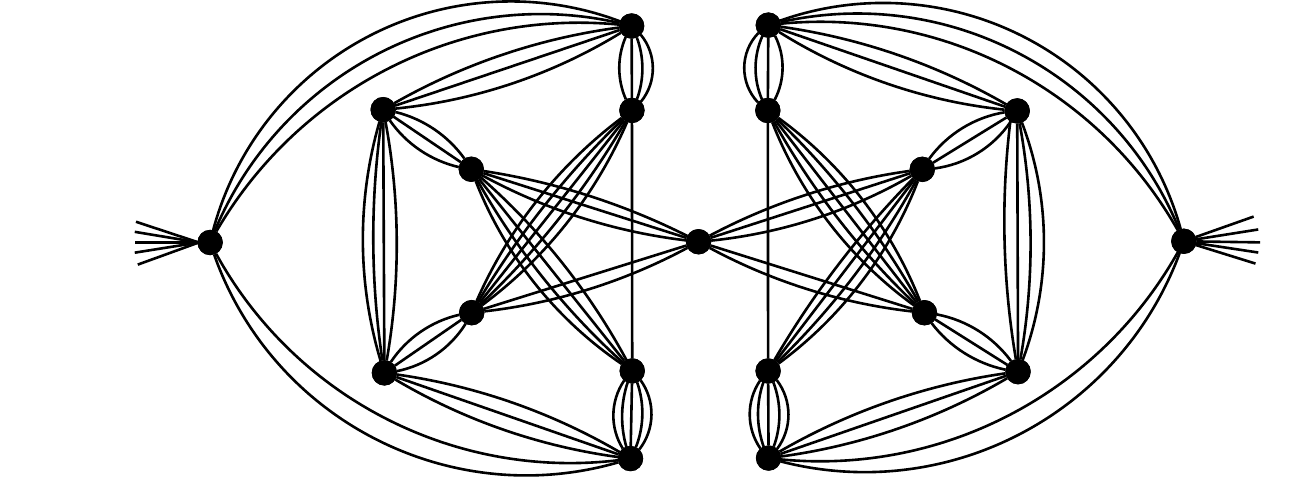}}%
    \put(0.0972923,0.16730601){\color[rgb]{0,0,0}\makebox(0,0)[rt]{\lineheight{0.31250054}\smash{\begin{tabular}[t]{r}$E^1_2$\end{tabular}}}}%
    \put(1.00182235,0.1688183){\color[rgb]{0,0,0}\makebox(0,0)[rt]{\lineheight{0.31250054}\smash{\begin{tabular}[t]{r}$E^2_2$\end{tabular}}}}%
    \put(0.55906724,0.20525762){\color[rgb]{0,0,0}\makebox(0,0)[rt]{\lineheight{0.31250054}\smash{\begin{tabular}[t]{r}$u_{Q_2}$\end{tabular}}}}%
  \end{picture}%
\endgroup%

\caption{The subgraph $Q_2$ and the edge-sets $E^1_2$ and $E^2_2$.}
\label{Fig:Q_2}
\end{figure}
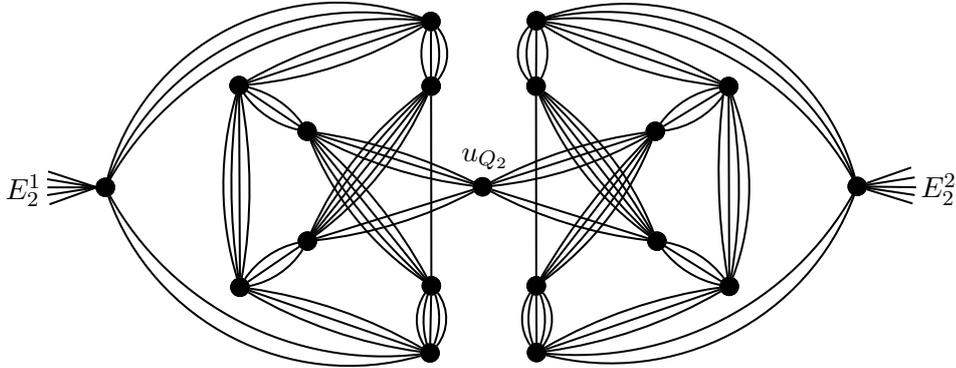

The following lemma is similar to Lemma 2.5 in \cite{Mattiolo2022HighlyER} ($Q_k$ is different), and it can be proved analogously. In order to keep the paper self-contained, we present the proof here.
	
	\begin{lem}\label{lem:Qk}
Let $G$ be a graph that contains $Q_k$ as an induced subgraph. Let $\ca N =\{N_1,\dots, N_{4k}\}$ be a set of pairwise disjoint perfect matchings of $G$ and let $N= \bigcup_{i=1}^{4k} N_i$. If $\partial(V(Q_k)) = E_k^{1} \cup E_k^{2}$, then $$\vert E_k^{1} \cap N \vert= \vert E_k^{2} \cap N \vert =2k.$$
	\end{lem}

	\begin{proof}
Every perfect matching of $G$ intersects $\partial(V(Q_k))$ precisely once since $|V(Q_k)|$ is odd and $\{v_1^1, v_1^2\}$ is a $2$-vertex-cut. It remains to show that $E_k^i$ intersects precisely $2k$ elements of $\ca N$. Recall that $Q_k$ is constructed by using two copies of $ P + \sum_{M\in \ca M} M$, where
$$\ca  M = \{\underbrace{M_0,M_1,M_2,\dots,M_0,M_1,M_2}_{ k\text{ times}}, \underbrace{M_5,\dots,M_5}_{ k-1\text{ times}}\}.$$
We argue by contradiction. Without loss of generality, suppose that $\vert E_k^{1} \cap N \vert < 2k$, which is equivalent to $\vert E_k^{2} \cap N \vert > 2k$. 
Every perfect matching of $\ca N$ that intersects $E_k^1$ also intersects 
the set $[u_{Q_k}, V(P_k^2)]_G$, and vice versa. 
Consequently, the existence of $\ca N$ implies that there is a set $\ca N'$ of $4k$ pairwise disjoint perfect matchings in $P_k^{1}$ such that $\ca M \nsubseteq \ca N'_P$, a contradiction to Lemma \ref{lem:P+matchings}. Hence, $\vert E_k^{1} \cap N \vert= \vert E_k^{2} \cap N \vert =2k$.
	\end{proof}

\subsection{The graph $S_k$}

For every $k \geq 1$, let $S_k$ be the graph with vertex set $\{x_i,y_i,z_i,w : i \in \{1,\ldots,4k+2\}\}$ and edge set $A_k \cup k B_k \cup (k+1)C_k \cup (2k+1)(D_k \cup E_k)$ where
\begin{align*}
A_k &=\{wz_i : i \in \{1,\ldots,4k+2\}\},\\
B_k &=\{z_i x_i, z_i y_i : i \in \{1,\ldots,4k+2\}\},\\
C_k &=\{x_i y_i : i \in \{1,\ldots,4k+2\}\},\\
D_k &=\{y_i x_{i+1} : i \in \{1,\ldots,4k+2\}\},\\
E_k &=\{z_i z_{i+2k+1} : i \in \{1,\ldots,2k+1\}\},
\end{align*}
and the indices are added modulo $4k+2$, see Figure~\ref{Fig:S_1andS_2}.

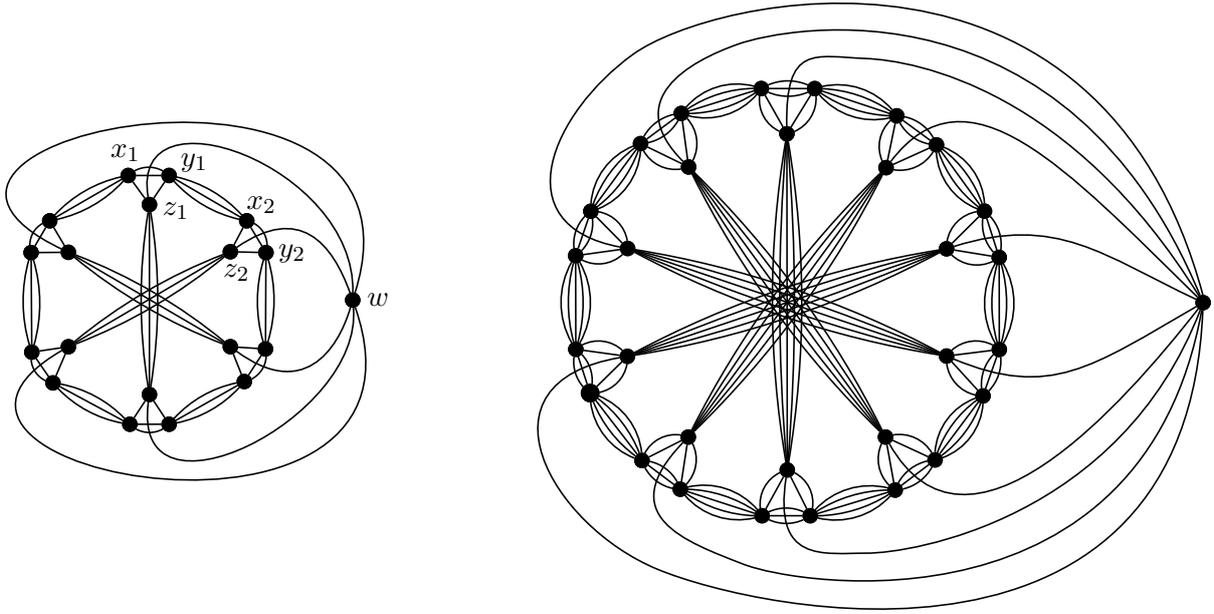
\begin{figure}[htbp]
\centering
\begingroup%
  \makeatletter%
  \providecommand\color[2][]{%
    \errmessage{(Inkscape) Color is used for the text in Inkscape, but the package 'color.sty' is not loaded}%
    \renewcommand\color[2][]{}%
  }%
  \providecommand\transparent[1]{%
    \errmessage{(Inkscape) Transparency is used (non-zero) for the text in Inkscape, but the package 'transparent.sty' is not loaded}%
    \renewcommand\transparent[1]{}%
  }%
  \providecommand\rotatebox[2]{#2}%
  \newcommand*\fsize{\dimexpr\f@size pt\relax}%
  \newcommand*\lineheight[1]{\fontsize{\fsize}{#1\fsize}\selectfont}%
  \ifx\svgwidth\undefined%
    \setlength{\unitlength}{449.6716321bp}%
    \ifx\svgscale\undefined%
      \relax%
    \else%
      \setlength{\unitlength}{\unitlength * \real{\svgscale}}%
    \fi%
  \else%
    \setlength{\unitlength}{\svgwidth}%
  \fi%
  \global\let\svgwidth\undefined%
  \global\let\svgscale\undefined%
  \makeatother%
  \begin{picture}(1,0.50763425)%
    \lineheight{1}%
    \setlength\tabcolsep{0pt}%
    \put(0,0){\includegraphics[width=\unitlength,page=1]{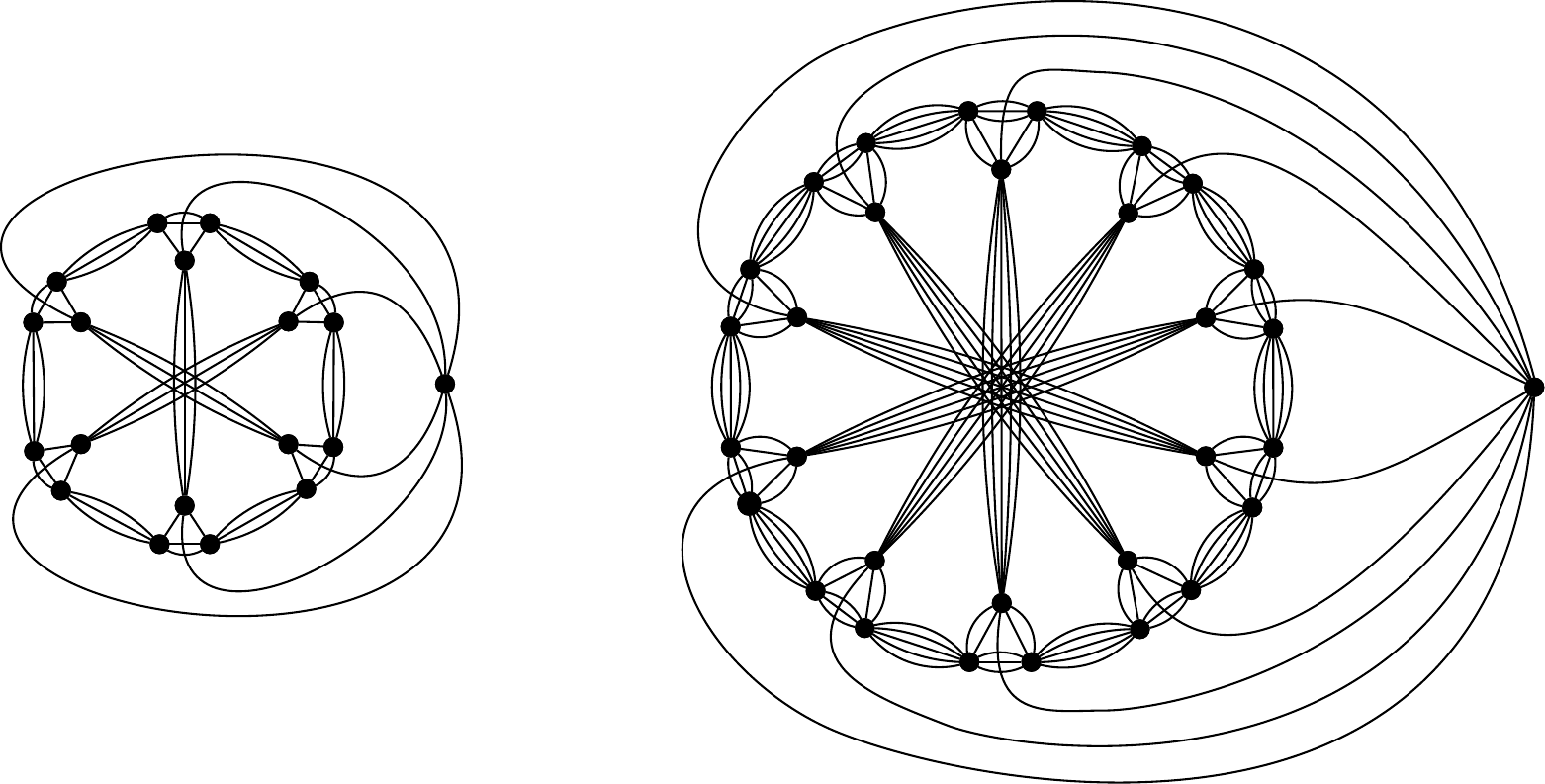}}%
    \put(0.11190181,0.37856123){\color[rgb]{0,0,0}\makebox(0,0)[rt]{\lineheight{0.3125006}\smash{\begin{tabular}[t]{r}$x_1$\end{tabular}}}}%
    \put(0.22474243,0.33502578){\color[rgb]{0,0,0}\makebox(0,0)[rt]{\lineheight{0.3125006}\smash{\begin{tabular}[t]{r}$x_2$\end{tabular}}}}%
    \put(0.24912034,0.29511828){\color[rgb]{0,0,0}\makebox(0,0)[rt]{\lineheight{0.3125006}\smash{\begin{tabular}[t]{r}$y_2$\end{tabular}}}}%
    \put(0.20268256,0.27872191){\color[rgb]{0,0,0}\makebox(0,0)[rt]{\lineheight{0.3125006}\smash{\begin{tabular}[t]{r}$z_2$\end{tabular}}}}%
    \put(0.16799534,0.37087192){\color[rgb]{0,0,0}\makebox(0,0)[rt]{\lineheight{0.3125006}\smash{\begin{tabular}[t]{r}$y_1$\end{tabular}}}}%
    \put(0.15232449,0.33038004){\color[rgb]{0,0,0}\makebox(0,0)[rt]{\lineheight{0.3125006}\smash{\begin{tabular}[t]{r}$z_1$\end{tabular}}}}%
    \put(0.31822508,0.25405179){\color[rgb]{0,0,0}\makebox(0,0)[rt]{\lineheight{0.3125006}\smash{\begin{tabular}[t]{r}$w$\end{tabular}}}}%
  \end{picture}%
\endgroup%

\caption{The graphs $S_1$ (left) and $S_2$ (right).}
\label{Fig:S_1andS_2}
\end{figure}

\begin{lem}\label{lem:S_k_connectivity}
For all $k \geq 1$, $S_k$ is $(4k+2)$-edge-connected and $(4k+2)$-regular.
\end{lem}

\begin{proof}
By definition, $S_k$ is $(4k+2)$-regular. Let $X \subset V(G)$ be a non-empty set. First, we consider the case that there are two vertices $u,v \in \{x_i, y_i : i \in \{1, \ldots, 4k+2\}\}$ such that $X$ contains exactly one of them. Clearly, there are $4k+2$ pairwise edge-disjoint $uv$-paths in $S_k$, which only contain edges of $kB_k \cup (k+1)C_k \cup 
(2k+1)D_k$. Hence, $\vert \partial_{S_k}(X) \vert \geq 4k+2$. Therefore, without loss of generality we may assume $\{x_i, y_i : i \in \{1, \ldots, 4k+2\}\} \cap X = \emptyset$. Since $S_k$ is $(4k+2)$-regular and $\mu_{S_k}(u,v) \leq 2k+1$ for every $u,v \in V(S_k)$, we have $\vert X \vert \notin \{1, 2\}$. Hence, $X$ either contains at least three vertices of $\{z_i : i \in \{1, \ldots, 4k+2\}\}$ or $w$ and exactly two vertices of $\{z_i : i \in \{1, \ldots, 4k+2\}\}$. In the first case, $\partial(X)$ contains at least $6k$ edges of 
$kB_k$. In the second case, $\partial(X)$ contains $4k$ edges of $A_k$ and at least $4k$ edges of $kB_k$, which completes the proof.
\end{proof}

\subsection{The graph $G_k$}

For every $k \geq 1$, let $G_k$ be the graph obtained from $S_k$ as follows. First, remove all edges of $(2k+1)(D_k \cup E_k)$. Then, for every edge $e=uv \in D_k \cup E_k$, add a copy $Q_k^e$ of $Q_k$, connect $u$ with the vertex corresponding to $v_1^1$ by $2k+1$ new parallel edges and connect $v$ with the vertex corresponding to $v_2^1$ by $2k+1$ new parallel edges, see Figure~\ref{Fig:G_1}.

\begin{figure}
	\centering
	\includegraphics[scale=0.35]{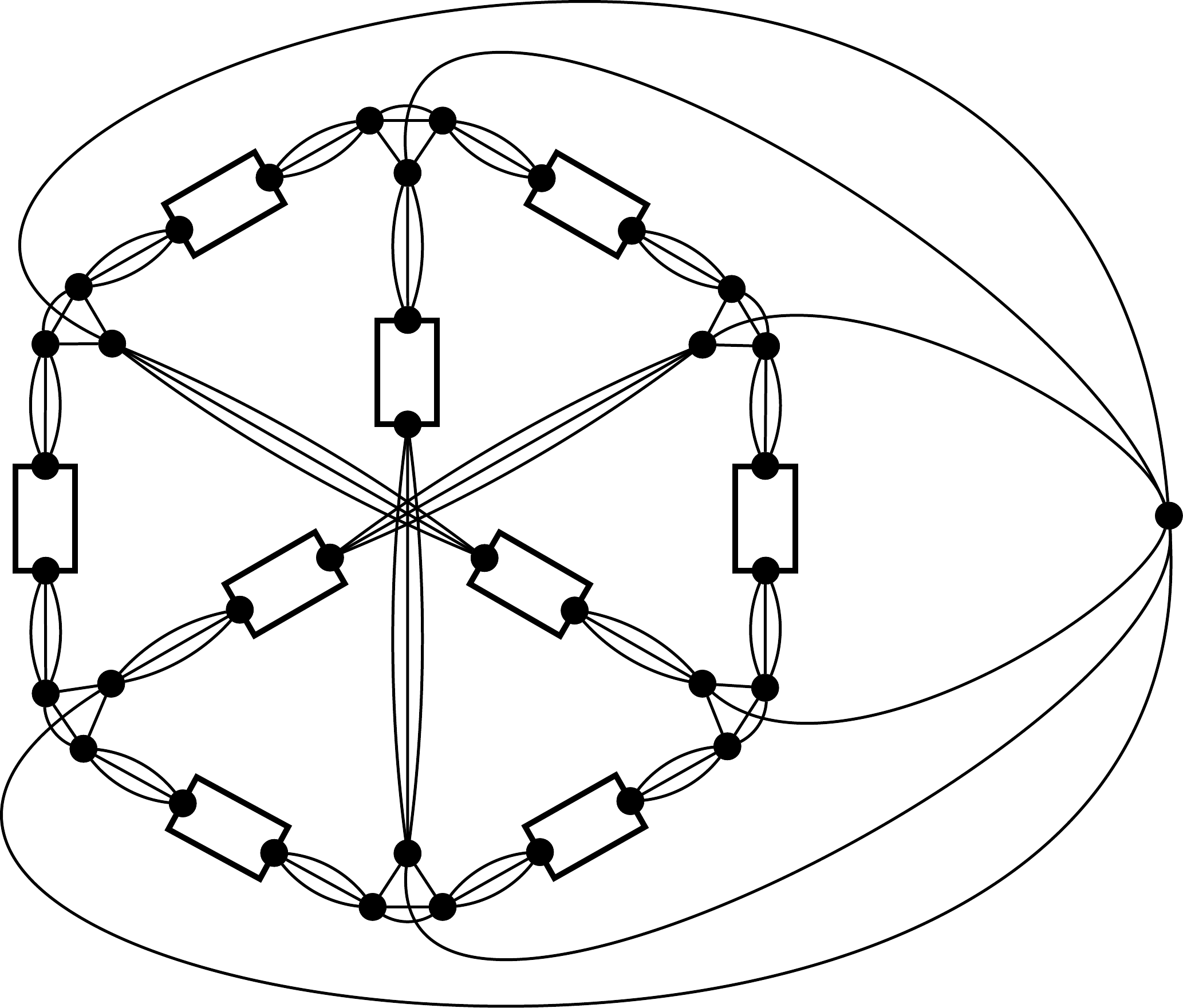}
	\caption{The graph $G_1$, where the boxes are copies of $Q_1$.}\label{Fig:G_1}
\end{figure}

In order to prove that $G_k$ has the desired properties, we need the following two observations.

\begin{obs}\label{1edgeobs}
Let $G$ be a graph and let $u,v \in V(G)$ with $\mu_G(u,v)=t$. Let $H$ be the graph obtained from $G$ by identifying $u$ and $v$ to a (new) vertex $w$ and removing all resulting loops. If $G$ is $2t$-edge-connected and $2t$-regular, then $H$ is $2t$-edge-connected and $2t$-regular.
\end{obs}

\begin{proof}
Assume that $G$ is $2t$-edge-connected and $2t$-regular. Since $\mu_G(u,v)=t$, it follows that $\partial_G(\{u,v\})=2t$ and hence, $H$ is $2t$-regular. Let $X \subset V(H)$ be a non-empty set. If $w \in X$, then $\vert \partial_{H}(X) \vert = \vert \partial_G(X\setminus \{w\} \cup \{u,v\}) \vert\ge 2t$. If $w \notin X$, then $\vert \partial_{H}(X) \vert = \vert \partial_G(X) \vert\ge 2t$.
\end{proof}

\begin{obs}\label{2edgeobs}
Let $G$ and $G'$ be two graphs and let $u,v \in V(G)$ and $u',v' \in V(G')$ such that $\mu_G(u,v)=\mu_{G'}(u',v')=t$. Let $H$ be the graph obtained from $G$ and $G'$ as follows. Remove the $t$ parallel edges between $u$ and $v$ and the $t$ parallel edges between $u'$ and $v'$. Add $t$ parallel edges between $u$ and $u'$ and $t$ parallel edges between $v$ and $v'$. If $G$ and $G'$ are $2t$-edge-connected and $2t$-regular, then $H$ is $2t$-edge-connected and $2t$-regular.
\end{obs}

\begin{proof}
Clearly, $H$ is $2t$-regular. Note that $G-[u,v]_G$ and $G'-[u',v']_{G'}$ are $t$-edge-connected. 
Let $X \subset V(H)$ be a non-empty set. 

	{\bf Case 1}. $X \cap V(G) = V(G)$ or $X \cap V(G') = V(G')$. 
	
	Say $V(G) \subseteq X$, 
then $\partial_{H}(X)$ contains either
(i) $[u,u']_H$ and $[v,v']_{H}$ or (ii) one of $[u,u']_H$ and $[v,v']_{H}$ and a $t_1$-edge-cut
of $G'-[u',v']_{G'}$ with $t_1 \geq t$ or (iii) a $t_2$-edge-cut of $G'-[u',v']_{G'}$ with
$t_2 \geq 2t$. 

	{\bf Case 2}. $X \cap V(G) \neq V(G)$ and $X \cap V(G') \neq V(G')$.

If $X \cap V(G) \neq \emptyset$ and $X \cap V(G') \neq \emptyset$, then $\vert \partial_{H}(X) \vert \geq \vert \partial_G(X \cap V(G)) \vert-t + \vert \partial_{G'}(X \cap V(G')) \vert-t \geq 2t$. If $X \cap V(G') = \emptyset$, then $\vert \partial_{H}(X) \vert \geq \vert \partial_G(X) \vert \geq 2t$.
\end{proof}

\begin{theo}\label{G_k}
For all $k \geq 1$, $G_k$ is a $(4k+2)$-edge-connected $(4k+2)$-graph without $4k$ pairwise disjoint perfect matchings. 
\end{theo}

\begin{proof}
By Lemma \ref{lem:P+matchings_Connectivity}, $P_k$ is $(4k+2)$-edge-connected and $(4k+2)$-regular. Hence, by Observations \ref{1edgeobs} and \ref{2edgeobs}, the graph $Q_k+ (2k+1)\{v_1^1 v_1^2\}$ is $(4k+2)$-edge-connected and $(4k+2)$-regular. Thus, $G_k$ is $(4k+2)$-edge-connected and $(4k+2)$-regular by Observation~\ref{2edgeobs} again. Furthermore, the order of $G_k$ is $\vert V(S_k) \vert + 19 \vert D_k \cup E_k \vert$, which is even. Suppose to the contrary that $G_k$ has $4k$ pairwise disjoint perfect matchings. Let $N \subseteq E(G_k)$ be the union of them and let $wz_i \in N$. Lemma \ref{lem:Qk} implies $\vert N \cap \partial_{G_k}(\{x_i,y_i,z_i\}) \vert=6k+1$. On the other hand, every perfect matching contains an odd number of edges of $\partial_{G_k}(\{x_i,y_i,z_i\})$ by Observation \ref{odd-insect}. Therefore, $\vert N \cap \partial_{G_k}(\{x_i,y_i,z_i\}) \vert$ is even, a contradiction.
\end{proof}

Theorem \ref{G_k} implies that $m(r) \leq r-3$ if $r \equiv 2 \mod 4$. Thus,
Theorem \ref{4k+1con} and Corollary \ref{cor: disprove Thomassen even} are proved.

\section{Equivalences for statements on the existence of a $k$-PDPM  }\label{sec:Equivalences}

The graph $G_k$ from the previous section has many $2$-vertex-cuts. 
The following observation shows that
such a construction will not apply for the odd case of Problem \ref{Thom-r-2}.

\begin{obs}\label{rr-3conn}
For odd $r \geq 3$, every $r$-edge-connected $r$-graph is $3$-connected.
\end{obs}

\begin{proof}
Let $G$ be an $r$-edge-connected $r$-graph. Clearly, $G$ is of even order and $2$-connected. Suppose that there are two vertices $v_1, v_2$ such that $G-\{v_1,v_2\}$ is not connected.
Then $G-\{v_1,v_2\}$ has exactly two components $A$ and $B$. Since the order of $G$ is even, $A$ and $B$ are either both of even order or both of odd order. In the first case, $\vert \partial(V(A)) \vert + \vert \partial(V(B)) \vert \leq \vert \partial(v_1) \vert + \vert \partial(v_2) \vert = 2r$. Since $A$ and $B$ are of even order, $\vert \partial(V(A)) \vert$ and $\vert \partial(V(B)) \vert$ are both even. Hence, it follows that either $ \vert \partial(V(A)) \vert < r$ or $ \vert \partial(V(B)) \vert < r$ since $r$ is odd. In the second case, $\vert \partial(V(A)\cup \{v_1\}) \vert + \vert \partial(V(B) \cup \{v_1\}) \vert = \vert \partial(v_1) \vert + \vert \partial(v_2) \vert = 2r$. Thus, $ \vert \partial(V(A) \cup \{v_1\}) \vert < r$ or $ \vert \partial(V(B) \cup \{v_1\}) \vert < r$ since $A$ and $B$ are of odd order. Therefore, both cases lead to a contradiction with the assumption that $G$ is $r$-edge-connected.
\end{proof}

We are going to prove some equivalent statements about the existence of a $k$-PDPM in  $r$-edge-connected $r$-graphs.

\begin{defi}\label{Replace-H}
Let  $G$ and $H$ be two disjoint $r$-regular graphs with $u\in V(G)$ and $v\in V(H)$. Let $(G,u)|(H,v)$ be the set of all graphs obtained by replacing the vertex $u$ of $G$ by $(H,v)$, that is, deleting $u$ from $G$ and $v$ from $H$, and then adding $r$ edges between $N_G(u)$ and $N_H(v)$ such that the resulting graph is regular.
\end{defi}

\begin{lem}\label{GU-rcon}
If  $G$ and $H$ are two disjoint $r$-edge-connected $r$-regular graphs with $u \in V(G)$ and $v\in V(H)$, then every graph in $(G,u)|(H,v)$ is $r$-regular and $r$-edge-connected. 
\end{lem}
\begin{proof}
Suppose to the contrary that there exists a graph $G'\in (G,u)|(H,v)$ with a set $X\subset V(G')$ such that $|\partial_{G'}(X)|\leq r-1$. 
 If $X\subseteq V(G-u)$ or $X\subseteq V(H-v)$, then $|\partial_G(X)|=|\partial_{G'}(X)|\leq r-1$ or $|\partial_H(X)|=|\partial_{G'}(X)|\leq r-1$, a contradiction. Hence, by symmetry, we assume  
$X\cap V(G-u)= X_1$,  $X\cap V(H-v)= X_2$, $X^c\cap V(G-u)= X_3$ and $X^c\cap V(H-v)= X_4$, where $X^c=V(G')-X$ and $X_i\neq \emptyset$ for each $i\in \{1,2,3,4\}$.  Since $|\partial_{G'}(X)|\leq r-1$, we have $|[X_1,X_3]_{G'}|\leq \lfloor\frac{r-1}{2}\rfloor$ or $|[X_2,X_4]_{G'}|\leq \lfloor\frac{r-1}{2}\rfloor$. It implies that $G-u$ or $H-v$ has an edge-cut of cardinality at most $\lfloor\frac{r-1}{2}\rfloor$, which contradicts the assumption that both $G$ and $H$ are $r$-edge-connected.
\end{proof}

 Let $e$ be an edge of a graph $G$ and $\ca M$ be a $k$-PDPM of $G$. We say that $\ca M$ \emph{contains} $e$ if there is an $N\in \ca M$ such that $e\in N.$ Otherwise, we say that $\ca M$ \emph{avoids} $e$. In what follows we show that if every  $r$-edge-connected $r$-graph has a $k$-PDPM, then every  $r$-edge-connected $r$-graph has a $k$-PDPM containing or avoiding a 
 fixed set of edges.

\begin{theo}\label{eq-st1}
 Let $r\geq4$ and $2\leq k\leq r-2$. The following statements are equivalent. \\
 (i) Every  $r$-edge-connected $r$-graph has a $k$-PDPM.\\
 (ii) For every $r$-edge-connected $r$-graph $G$ and every  $e\in E(G)$, there exists
  a $k$-PDPM of $G$ containing $e$.\\
 (iii) For every $r$-edge-connected $r$-graph $G$ and every  $e\in E(G)$, there exists a
$k$-PDPM of $G$ avoiding $e$.\\
 (iv)  For every $r$-edge-connected $r$-graph $G$, every $v\in V(G)$ and $e\in\partial_G(v)$, there are at least $s=r-\lfloor{\frac{r-k}{2}}\rfloor-1$ edges $e_1,\ldots,e_s$ in  $\partial_G(v) \setminus\{e\}$ such that, for each $i\in\{1,\ldots,s\}$,  there exists a $k$-PDPM of $G$ containing $e_i$ and $e$.
\end{theo}

 \begin{proof}
	Clearly,  each of $(ii)$, $(iii)$ and $(iv)$ implies $(i)$. Thus, it suffices to prove that $(i)$ implies $(ii)$; $(i)$ implies $(iii)$; and $(ii)$ implies $(iv)$.  
	
	$(i) \Rightarrow (ii), (iii)$. Assume that 
	statement $(i)$ is true and let $G$ be an $r$-edge-connected $r$-graph with an edge $vv_1$. We use the same construction for both implications. Let $C_{2r}=u_1u_2\ldots u_{2r}u_1$ be a circuit of length $2r$. Denote $U_o=\{u_i:i~\text{is odd}\}$ and $U_e=\{u_i:i~\text{is even} \}$.
	We construct a new graph $H$ from $C_{2r}$ as follows. 
	Replace each edge of $C_{2r}$ by $\frac{r-1}{2}$ parallel edges, if $r$ is odd, and replace the edge $u_iu_{i+1}$ ($u_ju_{j+1}$ and $u_{2r}u_1$, respectively) of $C_{2r}$ by $\frac{r}{2}$ ($\frac{r-2}{2}$, respectively) parallel edges for each $u_i\in U_o$ ($u_j\in U_e\setminus\{u_{2r}\}$, respectively),  if $r$ is even. Add two new vertices, denoted by $u$ and $u'$, such that $u$ is adjacent to each vertex in $U_o$ and $u'$ is adjacent to each vertex in $U_e$, see Figure \ref{fig1}. Clearly, $H$ is $r$-regular and $r$-edge-connected. 
	
	Let $I=\{ i  : i \in \{1, \ldots , 2r\}, i \text{ is odd}\}$ and for every $i \in I$ let $G^i$ be a copy of $G$, in which the vertices are labeled accordingly by using an upper index. For example, $v^i$ is the vertex of $G^i$ that corresponds to the vertex $v$ of $G$. Following the procedure described in Definition \ref{Replace-H}, we construct another new graph $H'$ from $H$ by successively replacing each vertex $u_i\in U_o$ of $H$ by $(G^i,v^i)$ such that for each $i \in I$ the vertex $v^i_1$ is adjacent to $u$ (see Figure \ref{fig2}). By Lemma \ref{GU-rcon}, $H'$ is $r$-regular and $r$-edge-connected. Note that $H'$
	is an $r$-graph since it is of even order.
	
	In order to prove statements $(ii)$ and $(iii)$ we observe the following. Let $M$ be an arbitrary perfect matching of $H'$ and for every $i \in I$, let $m_i=\vert \partial_{H'} (V(G^i-v^i)) \cap M \vert$. The set $M$ contains exactly one edge incident with $u$ and one edge incident with $u'$. Thus, by the construction of $H'$ we have $\sum_{i \in I} m_i=\vert M \cap \partial_{H'}(U_e)\vert=\vert I \vert$. Observation \ref{odd-insect} implies $m_i\geq 1$ and hence, $m_i=1$ for every $i \in I$. Thus, every perfect matching of $H'$ can be translated into a perfect matching of $G^i$ for each $i \in I$.
	
	Now, by statement $(i)$, $H'$ has a $k$-PDPM $\ca N$. Furthermore there are two integers $i,j \in I$ such that $\ca N$ contains $uv_1^{i}$ and avoids $uv_1^j$. By the above observation, the graph $G^i$ has a $k$-PDPM containing $v^iv_1^i$ and $G^j$ has a $k$-PDPM avoiding $v^jv_1^j$, which proves statements $(ii)$ and $(iii)$.

	$(ii) \Rightarrow (iv)$. Let $G$ be an $r$-edge-connected  $r$-graph and let $e_1=vv_1 \in E(G)$. Suppose $\vert \{ e \in \partial_G(v)\setminus \{e_1\}: \text{there exists a } k \text{-PDPM of } G \text{ containing } e, e_1\} \vert<s$. As a consequence, $\partial_G(v)\setminus \{e_1\}$ contains at least $t=r-1-(s-1)=\lfloor{\frac{r-k}{2}}\rfloor+1$ edges $e_2, \ldots, e_{t+1}$, such that for every $j \in \{2, \ldots, t+1\}$ there is no $k$-PDPM of $G$ containing $e_1$ and $e_j$. For each $j \in \{2, \ldots,t+1\}$ denote $e_j=vv_j$.
	
	Let $K_4$ be the complete graph of order $4$ and let $V(K_4)=\{u_1,u_2,u_3,u_4\}$. We construct a new $r$-regular graph $H$ from $K_4$ by replacing each edge of $\{u_1u_2, u_2u_3, u_3u_4, u_4u_1\}$ by $\frac{r-1}{2}$ parallel edges if $r$ is odd, and  replacing each edge of $\{u_1u_2, u_3u_4\}$ ($\{u_2u_3, u_4u_1\}$, respectively)  by $\frac{r}{2}$ ($\frac{r-2}{2}$, respectively) parallel edges  if $r$ is even, see Figure \ref{fig3}. Clearly, $H$ is $r$-edge-connected.
	
	For each $i\in\{1,3\}$, let $G^i$ be a copy of $G$ in which the vertices and edges are labeled accordingly by using an upper index and let $V^i=\{v^i_j: j \in\{2,\ldots, t+1\}\}$.
	Following the procedure in Definition \ref{Replace-H}, we construct another new graph $H'$ from $H$ by successively replacing each vertex $u_i\in\{u_1,u_3\}$ of $H$ by $(G^i,v^i)$ such that $v^1_1$ is adjacent to $v^3_1$ and 
	$[u_2,V^1\cup V^3]_{H'}$ contains as many edges as possible,  
	see Figure \ref{fig4}. The graph $H'$ is 
	$r$-regular and $r$-edge-connected by Lemma \ref{GU-rcon}. By statement $(ii)$, $H'$ has a $k$-PDPM $\mathcal{N}=\{N_1,\ldots,N_k\}$ containing $u_2u_4$. Clearly, $v^1_1v^3_1$  and $u_2u_4$ are in the same perfect matching of $\mathcal{N}$ and so each $N_i\in \mathcal{N}$ contains exactly one edge of $\partial_{H'} (V(G^1-v^1))$ and one edge of $\partial_{H'} (V(G^3-v^3))$ by Observation \ref{odd-insect}. Thus, $N_i\cap [u_2,V^1\cup V^3]_{H'}=\emptyset$ for each $i\in\{1,\ldots,k\}$. Now we consider the following two cases.

	{\bf Case 1}. $r$ is odd.
	
	Since $t=\lfloor{\frac{r-k}{2}}\rfloor+1\leq \frac{r-1}{2}$, the set $[u_2,V^i]_{H'}$ contains  $t$ edges for each $i\in\{1,3\}$ by the construction of $H'$.
	Note that $N_i\cap [u_2,V^1\cup V^3]_{H'}=\emptyset$ for each $i\in\{1,\ldots,k\}$.
	Hence,  the $k$-PDPM $\mathcal{N}$ of $H'$ contains at most $r-2t=r-2(\lfloor{\frac{r-k}{2}}\rfloor+1)\leq r-2(\frac{r-k-1}{2}+1)=k-1$ edges in $\partial_{H'}(u_2)$, a contradiction.  
	
	{\bf Case 2}. $r$ is even.
	
	{\bf Case 2.1}.  $k=2$.
	
	Since $t=\lfloor{\frac{r-2}{2}}\rfloor+1=\frac{r}{2}$,  the set $[u_2,V^1\cup V^3]_{H'}$ contains  $2t-1=r-1$ edges. Hence, the $k$-PDPM $\mathcal{N}$ of $H'$ contains at most $r-(2t-1)=1$ edges in $\partial_{H'}(u_2)$, a contradiction.   	
	
	{\bf Case 2.2}.  $k > 2$.
	
	Since $t=\lfloor{\frac{r-k}{2}}\rfloor+1\leq \frac{r-4}{2}+1=\frac{r}{2}-1$, we have that $[u_2,V^i]_{H'}$ contains  $t$ edges  for each $i\in\{1,3\}$ by the construction of $H'$. Hence, 
	the $k$-PDPM $\mathcal{N}$ of $H'$ contains at most $r-2t=r-2(\lfloor{\frac{r-k}{2}}\rfloor+1)\leq r-2(\frac{r-k-1}{2}+1)=k-1$ edges in $\partial_{H'}(u_2)$, a contradiction again.
\end{proof}

\begin{figure}[htbp]
\centering
\subfigure[$r=5$]{
\begin{minipage}[t]{8.5cm}
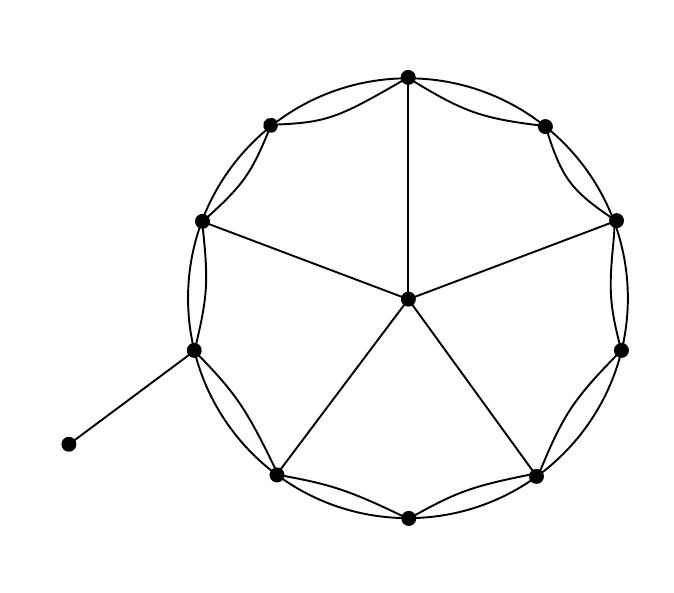
\end{minipage}
}
\subfigure[$r=4$]{
\begin{minipage}[t]{6.5cm}
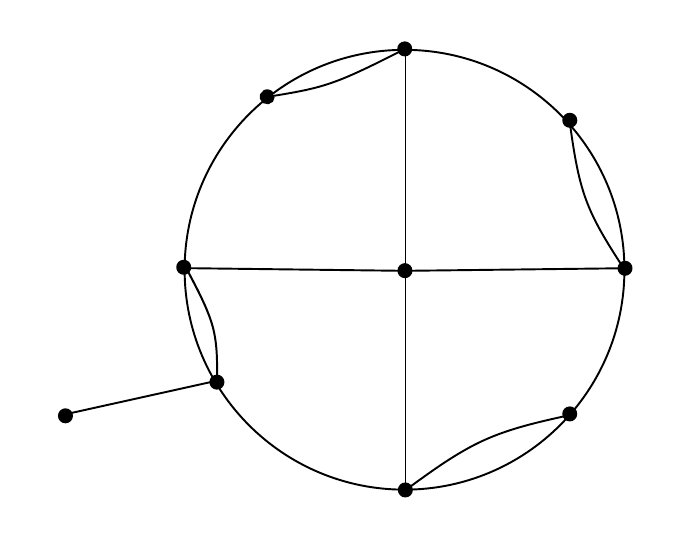
\end{minipage}
}
\caption{Two examples for the graph $H$ obtained from $C_{2r}$ as in the proof of Theorem \ref{eq-st1}.}
\label{fig1}
\end{figure}

\begin{figure}[htbp]
\centering
\subfigure[$r=5$ ]{
\begin{minipage}[t]{8.5cm}
\input{fig2a_korr.pdf_tex}
\end{minipage}
}
\subfigure[$r=4$]{
\begin{minipage}[t]{6.5cm}
\input{fig2b_korr.pdf_tex}
\end{minipage}
}
\caption{Two examples for the graph $H'$ obtained from $G$  and $H$ as in the proof of Theorem \ref{eq-st1}.}
\label{fig2}
\end{figure}

\begin{figure}[htbp]
\centering
\subfigure[$r=5$]{
\begin{minipage}[t]{7.5cm}
\centering
\begingroup%
  \makeatletter%
  \providecommand\color[2][]{%
    \errmessage{(Inkscape) Color is used for the text in Inkscape, but the package 'color.sty' is not loaded}%
    \renewcommand\color[2][]{}%
  }%
  \providecommand\transparent[1]{%
    \errmessage{(Inkscape) Transparency is used (non-zero) for the text in Inkscape, but the package 'transparent.sty' is not loaded}%
    \renewcommand\transparent[1]{}%
  }%
  \providecommand\rotatebox[2]{#2}%
  \newcommand*\fsize{\dimexpr\f@size pt\relax}%
  \newcommand*\lineheight[1]{\fontsize{\fsize}{#1\fsize}\selectfont}%
  \ifx\svgwidth\undefined%
    \setlength{\unitlength}{143.90827673bp}%
    \ifx\svgscale\undefined%
      \relax%
    \else%
      \setlength{\unitlength}{\unitlength * \real{\svgscale}}%
    \fi%
  \else%
    \setlength{\unitlength}{\svgwidth}%
  \fi%
  \global\let\svgwidth\undefined%
  \global\let\svgscale\undefined%
  \makeatother%
  \begin{picture}(1,1.1894751)%
    \lineheight{1}%
    \setlength\tabcolsep{0pt}%
    \put(0,0){\includegraphics[width=\unitlength,page=1]{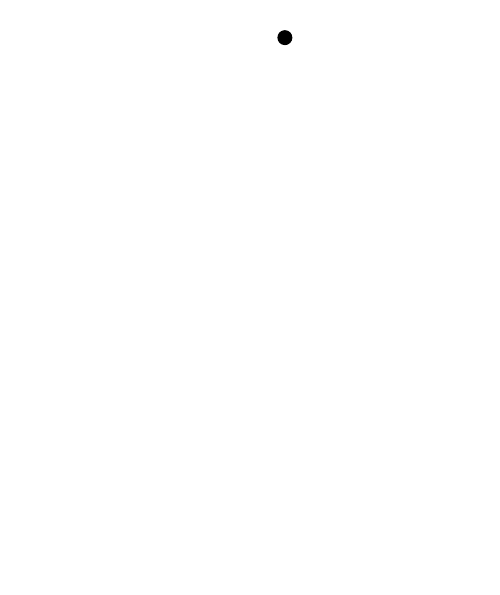}}%
    \put(0.65582655,1.147239){\color[rgb]{0,0,0}\makebox(0,0)[rt]{\lineheight{0.3125}\smash{\begin{tabular}[t]{r}$u_1$\end{tabular}}}}%
    \put(0.21718692,0.57387442){\color[rgb]{0,0,0}\makebox(0,0)[rt]{\lineheight{0.3125}\smash{\begin{tabular}[t]{r}$u_2$\end{tabular}}}}%
    \put(1.00464163,0.57302413){\color[rgb]{0,0,0}\makebox(0,0)[rt]{\lineheight{0.3125}\smash{\begin{tabular}[t]{r}$u_4$\end{tabular}}}}%
    \put(0.65697545,0.01311056){\color[rgb]{0,0,0}\makebox(0,0)[rt]{\lineheight{0.3125}\smash{\begin{tabular}[t]{r}$u_3$\end{tabular}}}}%
    \put(0,0){\includegraphics[width=\unitlength,page=2]{Fig3a.pdf}}%
  \end{picture}%
\endgroup%

\end{minipage}
}
\subfigure[$r=4$]{
\begin{minipage}[t]{7.5cm}
\centering
\begingroup%
  \makeatletter%
  \providecommand\color[2][]{%
    \errmessage{(Inkscape) Color is used for the text in Inkscape, but the package 'color.sty' is not loaded}%
    \renewcommand\color[2][]{}%
  }%
  \providecommand\transparent[1]{%
    \errmessage{(Inkscape) Transparency is used (non-zero) for the text in Inkscape, but the package 'transparent.sty' is not loaded}%
    \renewcommand\transparent[1]{}%
  }%
  \providecommand\rotatebox[2]{#2}%
  \newcommand*\fsize{\dimexpr\f@size pt\relax}%
  \newcommand*\lineheight[1]{\fontsize{\fsize}{#1\fsize}\selectfont}%
  \ifx\svgwidth\undefined%
    \setlength{\unitlength}{143.85217725bp}%
    \ifx\svgscale\undefined%
      \relax%
    \else%
      \setlength{\unitlength}{\unitlength * \real{\svgscale}}%
    \fi%
  \else%
    \setlength{\unitlength}{\svgwidth}%
  \fi%
  \global\let\svgwidth\undefined%
  \global\let\svgscale\undefined%
  \makeatother%
  \begin{picture}(1,1.20539746)%
    \lineheight{1}%
    \setlength\tabcolsep{0pt}%
    \put(0,0){\includegraphics[width=\unitlength,page=1]{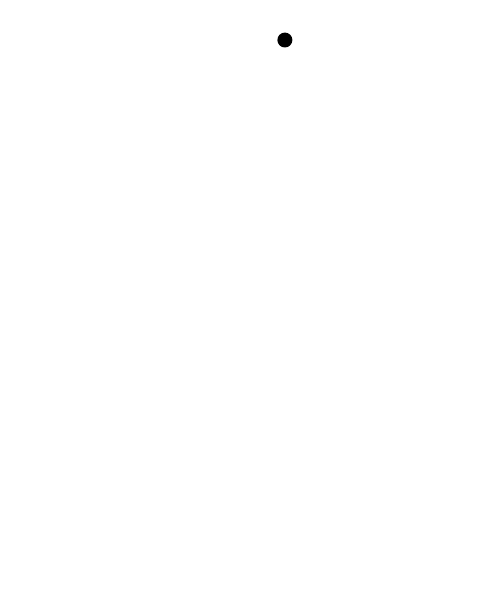}}%
    \put(0.65432304,1.16314488){\color[rgb]{0,0,0}\makebox(0,0)[rt]{\lineheight{0.3125}\smash{\begin{tabular}[t]{r}$u_1$\end{tabular}}}}%
    \put(0.21780938,0.58324065){\color[rgb]{0,0,0}\makebox(0,0)[rt]{\lineheight{0.3125}\smash{\begin{tabular}[t]{r}$u_2$\end{tabular}}}}%
    \put(1.00464344,0.58197569){\color[rgb]{0,0,0}\makebox(0,0)[rt]{\lineheight{0.3125}\smash{\begin{tabular}[t]{r}$u_4$\end{tabular}}}}%
    \put(0.65768883,0.01311568){\color[rgb]{0,0,0}\makebox(0,0)[rt]{\lineheight{0.3125}\smash{\begin{tabular}[t]{r}$u_3$\end{tabular}}}}%
    \put(0,0){\includegraphics[width=\unitlength,page=2]{Fig3b.pdf}}%
  \end{picture}%
\endgroup%

\end{minipage}
}
\caption{Two examples for the graph $H$ obtained from $K_4$ as in the proof of Theorem \ref{eq-st1}.}
\label{fig3}
\end{figure}

\begin{figure}[htbp]
\centering
\subfigure[$r=5$]{
\begin{minipage}[t]{7.5cm}
\centering
\begingroup%
  \makeatletter%
  \providecommand\color[2][]{%
    \errmessage{(Inkscape) Color is used for the text in Inkscape, but the package 'color.sty' is not loaded}%
    \renewcommand\color[2][]{}%
  }%
  \providecommand\transparent[1]{%
    \errmessage{(Inkscape) Transparency is used (non-zero) for the text in Inkscape, but the package 'transparent.sty' is not loaded}%
    \renewcommand\transparent[1]{}%
  }%
  \providecommand\rotatebox[2]{#2}%
  \newcommand*\fsize{\dimexpr\f@size pt\relax}%
  \newcommand*\lineheight[1]{\fontsize{\fsize}{#1\fsize}\selectfont}%
  \ifx\svgwidth\undefined%
    \setlength{\unitlength}{155.73688749bp}%
    \ifx\svgscale\undefined%
      \relax%
    \else%
      \setlength{\unitlength}{\unitlength * \real{\svgscale}}%
    \fi%
  \else%
    \setlength{\unitlength}{\svgwidth}%
  \fi%
  \global\let\svgwidth\undefined%
  \global\let\svgscale\undefined%
  \makeatother%
  \begin{picture}(1,1.29407035)%
    \lineheight{1}%
    \setlength\tabcolsep{0pt}%
    \put(0.70311426,1.25504219){\color[rgb]{0,0,0}\makebox(0,0)[rt]{\lineheight{0.3125}\smash{\begin{tabular}[t]{r}$G^1-v^1$\end{tabular}}}}%
    \put(0.26610845,0.62009863){\color[rgb]{0,0,0}\makebox(0,0)[rt]{\lineheight{0.3125}\smash{\begin{tabular}[t]{r}$u_2$\end{tabular}}}}%
    \put(1.00428908,0.61987534){\color[rgb]{0,0,0}\makebox(0,0)[rt]{\lineheight{0.3125}\smash{\begin{tabular}[t]{r}$u_4$\end{tabular}}}}%
    \put(0.73582661,0.00754979){\color[rgb]{0,0,0}\makebox(0,0)[rt]{\lineheight{0.3125}\smash{\begin{tabular}[t]{r}$G^3-v^3$\end{tabular}}}}%
    \put(0,0){\includegraphics[width=\unitlength,page=1]{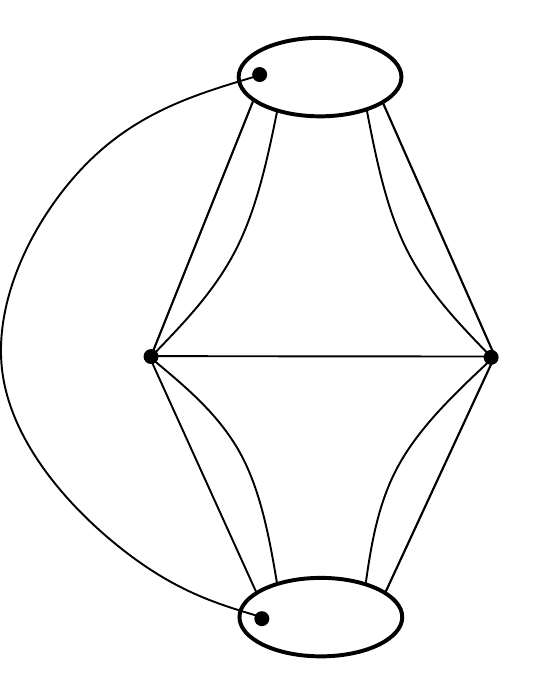}}%
    \put(0.56982231,1.14097895){\color[rgb]{0,0,0}\makebox(0,0)[rt]{\lineheight{0.3125}\smash{\begin{tabular}[t]{r}$v_1^1$\end{tabular}}}}%
    \put(0.56641286,0.13502643){\color[rgb]{0,0,0}\makebox(0,0)[rt]{\lineheight{0.3125}\smash{\begin{tabular}[t]{r}$v_1^3$\end{tabular}}}}%
  \end{picture}%
\endgroup%

\end{minipage}
}
\subfigure[$r=4$]{
\begin{minipage}[t]{7.5cm}
\centering
\begingroup%
  \makeatletter%
  \providecommand\color[2][]{%
    \errmessage{(Inkscape) Color is used for the text in Inkscape, but the package 'color.sty' is not loaded}%
    \renewcommand\color[2][]{}%
  }%
  \providecommand\transparent[1]{%
    \errmessage{(Inkscape) Transparency is used (non-zero) for the text in Inkscape, but the package 'transparent.sty' is not loaded}%
    \renewcommand\transparent[1]{}%
  }%
  \providecommand\rotatebox[2]{#2}%
  \newcommand*\fsize{\dimexpr\f@size pt\relax}%
  \newcommand*\lineheight[1]{\fontsize{\fsize}{#1\fsize}\selectfont}%
  \ifx\svgwidth\undefined%
    \setlength{\unitlength}{155.08279006bp}%
    \ifx\svgscale\undefined%
      \relax%
    \else%
      \setlength{\unitlength}{\unitlength * \real{\svgscale}}%
    \fi%
  \else%
    \setlength{\unitlength}{\svgwidth}%
  \fi%
  \global\let\svgwidth\undefined%
  \global\let\svgscale\undefined%
  \makeatother%
  \begin{picture}(1,1.27843904)%
    \lineheight{1}%
    \setlength\tabcolsep{0pt}%
    \put(0.70735258,1.23924627){\color[rgb]{0,0,0}\makebox(0,0)[rt]{\lineheight{0.3125}\smash{\begin{tabular}[t]{r}$G^1-v^1$\end{tabular}}}}%
    \put(0.26301296,0.61111497){\color[rgb]{0,0,0}\makebox(0,0)[rt]{\lineheight{0.3125}\smash{\begin{tabular}[t]{r}$u_2$\end{tabular}}}}%
    \put(1.00430717,0.60667258){\color[rgb]{0,0,0}\makebox(0,0)[rt]{\lineheight{0.3125}\smash{\begin{tabular}[t]{r}$u_4$\end{tabular}}}}%
    \put(0.72296746,0.00758164){\color[rgb]{0,0,0}\makebox(0,0)[rt]{\lineheight{0.3125}\smash{\begin{tabular}[t]{r}$G^3-v^3$\end{tabular}}}}%
    \put(0,0){\includegraphics[width=\unitlength,page=1]{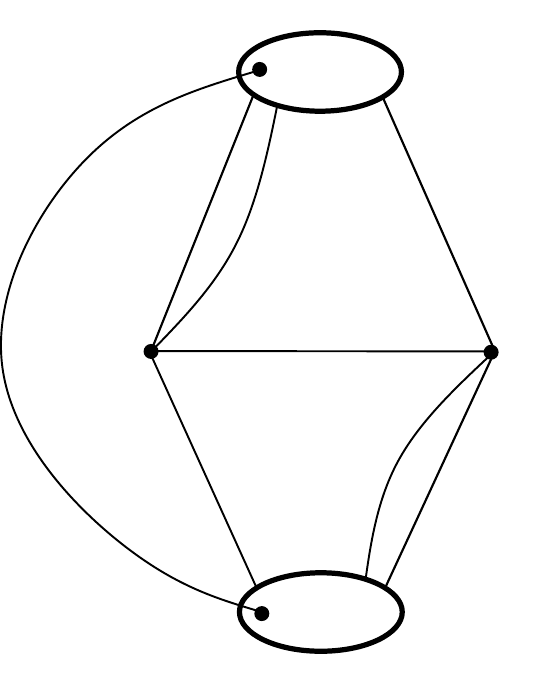}}%
    \put(0.56800766,1.13419249){\color[rgb]{0,0,0}\makebox(0,0)[rt]{\lineheight{0.3125}\smash{\begin{tabular}[t]{r}$v_1^1$\end{tabular}}}}%
    \put(0.56669291,0.11977899){\color[rgb]{0,0,0}\makebox(0,0)[rt]{\lineheight{0.3125}\smash{\begin{tabular}[t]{r}$v_1^3$\end{tabular}}}}%
  \end{picture}%
\endgroup%

\end{minipage}
}
\caption{Two examples for the graph $H'$ obtained from $G$  and $H$ as in the proof of Theorem \ref{eq-st1}.}
\label{fig4}
\end{figure}
 For the special case $k=r-2$, we can obtain a stronger result as follows.

\begin{theo}\label{eq-st3}
	Let $k\geq1$. The following statements are equivalent.\\
	(i) Every  $(2k+1)$-edge-connected $(2k+1)$-graph has a $(2k-1)$-PDPM.\\
	(ii) For every $(2k+1)$-edge-connected $(2k+1)$-graph $G$ and every $k$ edges sharing a common vertex, there exists a $(2k-1)$-PDPM of $G$ containing this $k$ edges.
\end{theo}

\begin{proof}
It suffices to prove that statement $(i)$ implies statement $(ii)$. Let $G$ be a $(2k+1)$-edge-connected $(2k+1)$-graph and let $v\in V(G)$ be a vertex with $\partial_G(v)=\{e_i:i \in \{1,\ldots,2k+1\}\} $. We show that there is a $(2k-1)$-PDPM of $G$ that contains the edges  $e_1,\ldots,e_{k}$.

	Denote $e_i=vv_i$ for each $i \in \{1,\ldots,2k+1\}$. Let $G^1$ be a copy of $G$ in which the vertices and edges are labeled accordingly by using an upper index. As described in Definition \ref{Replace-H}, construct a new graph $H$ from $G$ by replacing $v$ with $(G^1, v^1)$ such that the set of new edges is given by $\{v_{2k+1}v_{2k+1}^1\} \cup E_1 \cup E_2$, where $E_1=\{v_iv_{i+k}^1 : i \in \{1,\ldots,k\}\}$ and $E_2= \{v_i^1v_{i+k} : i \in \{1,\ldots,k\}\}$. By Lemma \ref{GU-rcon}, $H$ is $(2k+1)$-edge-connected and $(2k+1)$-regular. Thus, by statement $(i)$ and Theorem \ref{eq-st1} there is a $(2k-1)$-PDPM $\ca N$ of $H$ avoiding $v_{2k+1}v_{2k+1}^1$. By Observation \ref{odd-insect}, every perfect matching of $\ca N$ contains exactly one edge of $\partial_H(V(G) \setminus \{v\})$ and hence, $\ca N$ contains either every edge of $E_1$ or every edge of $E_2$. In the first case, $G$ has a $(2k-1)$-PDPM that contains $e_1,\ldots,e_{k}$; in the second case, $G^1$ has a $(2k-1)$-PDPM that contains $e_1^1,\ldots,e_{k}^1$. This proves statement $(ii)$.
\end{proof}

\section{$5$-graphs}\label{sec:disjoint_pm_5_reg_graphs}

We explore some consequences of the non-existence of $5$-edge-connected class $2$ $5$-graphs. The edge-connectivity may play a crucial role in this case as, by a result of Rizzi \cite{rizzi1999indecomposable}, there are poorly matchable $4$-edge-connected $5$-graphs.

Let $G$ be a cubic graph and let $\ca F = \{F_1,\dots, F_t\}$ be a multiset of subsets $F_i$ of $E(G)$. For an edge $e$ of $G$, we denote by $\nu_{\ca F}(e)$ the number of elements of $\ca F$ containing $e$. A \emph{Fan-Raspaud triple}, or $FR$-\emph{triple}, is a multiset $\ca T$ of three perfect matchings of $G$ such that $\nu_{\ca T}(e)\le 2$ for all $e\in E(G)$. A $5$-\emph{cycle double cover}, or $5$-CDC, is a multiset $\mathcal{C}$ of five cycles in $G$ such that, for every edge $e\in E(G)$, $\nu_{\mathcal{C}}(e)=2$. A \emph{Berge-Fulkerson cover}, or $BF$\emph{-cover}, is a multiset $\ca T$ of six perfect matchings of $G$ such that $\nu_{\ca T}(e) = 2$ for all $e\in E(G)$. We recall the following three well-known conjectures and a result from \cite{Kaiser_2-factors}.

\begin{con}[Fan-Raspaud Conjecture \cite{FAN1994133}]\label{FRC}
Every bridgeless cubic graph has an $FR$-triple.
\end{con}

\begin{con}[$5$-cycle double cover Conjecture, see \cite{C.-Q._Zhang_book}]\label{5CDC}
Every bridgeless cubic graph has a $5$-cycle double cover.
\end{con}

\begin{con}[Berge-Fulkerson Conjecture \cite{fulkerson1971blocking}]\label{BFC}
Every bridgeless cubic graph has a $BF$-cover.
\end{con}

\begin{theo}[Kaiser and \v Skrekovski \cite{Kaiser_2-factors}]\label{2-factor-theo}
	Every bridgeless cubic graph has a $2$-factor that intersects every edge-cut of cardinality $3$ and $4$. Moreover, any two adjacent edges can be extended to such a $2$-factor.
\end{theo}

As shown in \cite{MKRTCHYAN2019}, the following conjecture is equivalent to the Fan-Raspaud Conjecture.

\begin{con}[Mkrtchyan and Vardanyan \cite{MKRTCHYAN2019}]\label{strongFRC}
Let $G$ be a bridgeless cubic graph. For every $e \in E(G)$ and $i \in \{0,1,2\}$, 
there is an $FR$-triple $\ca T$ with $\nu_{\ca T}(e)=i$.
\end{con}

In the same paper, they also pointed out the following observation but without proof.  
To keep the paper self-contained, we present a short proof here.

\begin{obs}[Mkrtchyan and Vardanyan \cite{MKRTCHYAN2019}] \label{cotex-3con}
A minimum possible counterexample $G$ to Conjecture \ref{strongFRC} with respect to $|V(G)|$ is $3$-edge-connected.
\end{obs}

\begin{proof} 
Suppose that $G$ is a minimum counterexample to Conjecture \ref{strongFRC} with respect to $|V(G)|$. Then, there is $e\in E(G)$ and $i\in\{0,1,2\}$ such that no $FR$-triple $\ca T$ satisfies $\nu_{\ca T}(e)=i$. Suppose that there is a set $X\subseteq V(G)$ with $u,v\in X$ and $\partial_G(X)=\{ux,vy\}$. Let $H_1=G[X]+\{uv\}$ and $H_2=G-X+\{xy\}$. Notice that both $H_1$ and $H_2$ are bridgeless cubic graphs.
If $ e \in \{ux,vy\}$, since $|V(G)|$ is minimum, there is an $FR$-triple $\ca T_1$ of $H_1$ and an $FR$-triple $\ca T_2$ of $H_2$ such that $\nu_{\ca T_1}(uv)=\nu_{\ca T_2}(xy)=i$. Then $\ca T_1$ and $\ca T_2$ can be used to construct an $FR$-triple $\ca T$ of $G$ with $\nu_{\ca T}(e)=i$, a contradiction. Hence, without loss of generality we may assume $e \in E(H_1)$. Since $|V(G)|$ is minimum, there is an $FR$-triple $\ca T_1$ of $H_1$ and an $FR$-triple $\ca T_2$ of $H_2$ such that $\nu_{\ca T_1}(e)=i$ and $\nu_{\ca T_2}(xy)=\nu_{\ca T_1}(uv)$. Again, $\ca T_1$ and $\ca T_2$ can be used to construct an $FR$-triple $\ca T$ of $G$ with $\nu_{\ca T}(e)=i$, a contradiction.
\end{proof}

\subsection{Relation to the Fan-Raspaud Conjecture}

In this subsection we show that the Fan-Raspaud Conjecture is true if there is no poorly matchable $5$-edge-connected $5$-graph.

\begin{theo}\label{FR_implication}
If $m(5) \geq 2$, then Conjecture~\ref{strongFRC} is true.
\end{theo}

\begin{proof}
By contradiction, suppose that $m(5) \geq 2$ and Conjecture~\ref{strongFRC} is false. Let $G$ be a minimum counterexample to Conjecture \ref{strongFRC} with respect to $|V(G)|$. Then, there is an edge $e = uv$ of $G$ and an $i\in\{0,1,2\}$ such that no $FR$-triple $\ca T$ satisfies $\nu_{\ca T}(e)=i$. By Observation \ref{cotex-3con}, $G$ is $3$-edge-connected.

First, we consider the case $i=0$. By Theorem \ref{2-factor-theo} there is a $2$-factor $F$ of $G$ such that $e \in E(F)$ and $F$ intersects every edge-cut of cardinality $3$ and $4$. Let $H = G + E(F)$ and let $e'$ be the new edge parallel to $e$. Since $G$ is $3$-edge-connected, the graph $H$ is $5$-edge-connected by the choice of $F$. 
Since $m(5) \geq 2$, it follows with Theorem \ref{eq-st1} $(iv)$ that for each edge $e_0 \in \partial_H(v)\setminus\{e,e'\}$, there are at least three edges $e_1,e_2,e_3 \in \partial_H(v) \setminus \{e_0\}$ such that
for each $j \in \{1,2,3\}$ there exists a $2$-PDPM containing $e_j$ and $e_0$. This implies that $H$ has two disjoint perfect matchings $N_1$ and $N_2$ such that $e$ and $e'$ are in none of them. In the graph $G$,
let $N_{1}'$ and $N_{2}'$ be the perfect matchings corresponding to $N_1$ and $N_2$, respectively. Let $N_3=E(G)\setminus E(F)$. Since $N_1$ and $N_2$ are disjoint, every edge of $N_{1}' \cap N_{2}'$ belongs to $E(F)$, i.e.\ $\ca T=\{N_{1}',N_{2}',N_3\}$ is an $FR$-triple of $G$. Furthermore $\nu_{\ca T}(e)=0$, a contradiction.

Next suppose $i \in \{1,2\}$. By Theorem \ref{2-factor-theo} we can choose a $2$-factor $F$ of $G$ such that $e \notin E(F)$ and $F$ intersects every edge-cut of cardinality $3$ and $4$. Again, the graph $H$ defined by $H = G + E(F)$ is $5$-edge-connected. Since $m(5) \geq 2$, by statements $(ii)$ and $(iii)$ of Theorem \ref{eq-st1}, $H$ has two disjoint perfect matchings $N_1$ and $N_2$ such that $e$ is in exactly $i-1$ of them. Therefore, $\ca T=\{N_{1}',N_{2}',N_3\}$ is an $FR$-triple of $G$ with $\nu_{\ca T}(e)=i$ where $N_{1}'$ and $N_{2}'$ are the perfect matchings of $G$ that correspond to $N_1$ and $N_2$, respectively, and $N_3=E(G)\setminus E(F)$. This leads to a contradiction again.
\end{proof}

If $m(5) \geq 2$, then in particular every $5$-edge-connected $5$-graph with an underlying cubic graph has two disjoint perfect matchings. By adjusting Theorem \ref{eq-st1},  one can show the following strengthening of Theorem \ref{FR_implication} (for a sketch of the proof, see Appendix \ref{Prof_FR_implication_cubic_case}).

\begin{theo} \label{FR_implication_cubic_case}
If every $5$-edge-connected $5$-graph whose underlying graph is cubic has two disjoint perfect matchings, then Conjecture \ref{strongFRC} is true.
\end{theo}

\subsection{Relation to the $5$-cycle double cover Conjecture}

Now we focus on the consequences of the non-existence of  $5$-edge-connected class $2$ $5$-graphs. 
Let $k \geq 3$ be an integer. A $k$-\emph{wheel} $W_k$ is a $k$-circuit $C_k$ plus one additional 
vertex $w$ adjacent to all vertices of $C_k$.

\begin{theo}
	The following statements are equivalent.\\	
(i) Every  $5$-edge-connected $5$-graph is class $1$.\\
(ii) Every  $5$-edge-connected $5$-graph with an underlying cubic graph is class $1$.
 
\end{theo}
\begin{proof}
	The first statement implies trivially the second one. We prove now the other implication. Let $G$ be a  $5$-edge-connected $5$-graph.  For every vertex $v$ of $G$, let $W_5^v$ be a copy of the graph $W_5+E(C_5)$. Moreover, let $w^v$ and $C_5^v$ be the vertex and, respectively, the circuit of $W_5^v$ corresponding to $w$ and $C_5$ in $W_5$. Following the procedure described in Definition \ref{Replace-H}, successively replace every vertex $v$ of $G$ with $(W_5^v,w^v)$
	to obtain a new graph $H$, which is $5$-regular and $5$-edge-connected. Moreover, its underlying graph is cubic and so $H$ is class $1$ by statement $(ii)$. Hence, $H$ has a $5$-PDPM, denoted by ${\cal N}=\{N_1,\ldots,N_5\}$. Since $|V(C_5^v)|$ is odd, by Observation \ref{odd-insect}, we have that, for all $i\in\{1,\dots,5\}$, $|N_i\cap \partial_H(V(C_5^v))|=1$. Hence, the restriction $N_i'$ of $N_i$ to the graph $G$ is a perfect matching of $G$. Moreover, $\{N_1',\dots,N_5'\}$ is a $5$-PDPM of $G$. Therefore, $G$ is class $1$.
\end{proof}

It is well known that a counterexample of minimum order to Conjecture \ref{5CDC} is a cyclically $4$-edge-connected cubic class $2$ graph. 

\begin{theo}
	If  $m(5)=5$, then Conjecture \ref{5CDC} is true.
\end{theo}

\begin{proof}
Let $K$ be the graph obtained from a 4-wheel by doubling the edges of the outer circuit and of
one spoke. Note that $K$ has one vertex of degree $6$, which we denote by $w$, 
and four vertices of degree $5$.
	
	Let $G$ be a minimum counterexample to Conjecture \ref{5CDC} with respect to $|V(G)|$. Then, $G$ is cubic and cyclically $4$-edge-connected. Thus, the graph $2G=G+E(G)$ is $6$-edge-connected. For every vertex $v$ of $G$, let $K^v$ be a copy of $K$ and let $w^v$ be the vertex of $K^v$ corresponding to $w$ in $K$. Analogously to  Definition \ref{Replace-H}, let $H$ be the graph obtained by replacing each vertex $v$ of $2G$ by $(K^v,w^v)$, in such a way that parallel edges of $2G$ are incident with the same vertex of $K^v$. Then, $H$ is a  $5$-edge-connected $5$-graph and therefore, it has a  $5$-PDPM $\mathcal{N}=\{N_1,\ldots,N_5\}$. For every $v\in V(2G)$, there exist exactly three perfect matchings of $\mathcal{N}$, say $N'_1, N'_2, N'_3$, such that  $|N'_i\cap\partial_H(K^v-w^v)|=2$ for each $i\in\{1,2,3\}$. Hence, for every $j\in\{1,\dots,5\}$, the restriction of each $N_j\in \mathcal{N}$ on $G$ induces an even subgraph $C_j'$ of $G$. Moreover, we have $\nu_{\mathcal{C}}(e)=2$ for each $e\in E(G)$, where $\mathcal{C}=\{C_1',\dots,C_5'\}$. So $\mathcal{C}$ is a $5$-CDC of $G$.
\end{proof} 

\subsection{Relation to the Berge-Fulkerson Conjecture}

\begin{obs}\label{BF_implication}
If $m(5) = 5$, then Conjecture \ref{BFC} is true.
\end{obs}

\begin{proof}
Assume $m(5)=5$ and suppose that $G$ is a counterexample to the Berge-Fulkerson Conjecture such that the order of $G$ is minimum. Let $F$ be a $2$-factor of $G$. As shown in \cite{ReductionBFC}, $G$ is cyclically $5$-edge-connected and hence, $G+E(F)$ is $5$-edge-connected. Therefore, $G+E(F)$ has five pairwise disjoint perfect matchings. The corresponding five perfect matchings of $G$ and $E(G) \setminus E(F)$ are a $BF$-cover of $G$, a contradiction.
\end{proof}

\subsection{Properties of a minimum possible $5$-edge-connected class $2$ $5$-graph}

We are going to prove some structural properties of a smallest possible
$5$-edge-connected class $2$ $5$-graph. 
Let $G$ be a graph and $v \in V(G)$. A \emph{lifting} of $G$ at $v$ is the following operation. Remove two edges $vx$, $vy$ where $x \neq y$ and add a new edge $xy$. In this case we say $vx$ and $vy$ are lifted to $xy$.

\begin{theo}[Mader \cite{MaderLifting}]\label{LiftingTheo}
Let $G$ be a finite graph and let $v \in V(G)$ such that $d(v) \geq 4$, $\vert N(v) \vert \geq 2$ and $G-v$ is connected. There is a lifting of $G$ at $v$ such that, for every pair of distinct vertices $u,w \in V(G) \setminus \{v\}$, the number of edge-disjoint $uw$-paths in the resulting graph equals the number of edge-disjoint $uw$-paths in $G$.
\end{theo}

Statement $(ii)$ of the following theorem is already
mentioned in \cite{planar8-graphsclass1} for 
planar $r$-graphs without proof. 

\begin{theo}\label{smallest5regClass2graph}
Let $G$ be a $5$-edge-connected class $2$ $5$-graph such that the order of $G$ is as small as possible. The following statements hold.\\
(i) Every $5$-edge-cut of $G$ is trivial, i.e.\ if  $X \subset V(G)$ and $\vert \partial(X) \vert =5$, then $\vert X \vert =1$ or $\vert V(G) \setminus X \vert =1$. \\
(ii) Every $3$-vertex-cut is trivial, i.e.\ if $X\subset V(G)$, $\vert X \vert =3$ and $G-X$ is not connected, then one component of $G-X$ is a single vertex.
\end{theo}

\begin{proof}
$(i)$. The proof follows easily and is left to the reader.

$(ii)$. By contradiction, suppose that  $X=\{v_1,v_2,v_3\}\subset V(G)$ is a $3$-vertex-cut of $G$ such that none of the components of $G-X$ is a single vertex. 
By Observation \ref{rr-3conn} and the edge-connectivity of $G$, the graph $G-X$ has at most three components. First, we consider the case that $G-X$ has exactly three components. Denote the vertex sets of these three components  by $A$, $B$ and $C$. We have that $|\partial_G(S)|=5$, for each $S\in\{A,B,C\}$, and so $|A|=|B|=|C|=1$ by statement $(i)$, a contradiction. 

Next, we assume that $G-X$ has exactly two components
whose vertex sets are denoted by  $A$ and $B$. Since $G$ has even order, we may assume $\vert A \vert$ is odd and $\vert B \vert$ is even.
For each $i \in \{1,2,3\}$, set $n_i= |\partial_G(B) \cap \partial_G(v_i)|$ and let
\begin{align*}
a= \frac{1}{2} \left(n_1+n_2-n_3\right), \quad
b= \frac{1}{2} \left(-n_1+n_2+n_3\right), \quad
c= \frac{1}{2} \left(n_1-n_2+n_3\right).
\end{align*}
We have that $n_1+n_2+n_3=\vert \partial_G(B) \vert$ is even, since $\vert B \vert$ is even. Thus, all of $a,b,c$ are integers. Furthermore, $5 \leq \vert \partial_G(B \cup \{v_3\}) \vert = n_1+n_2+(5-n_3)$ and hence, $a \geq 0$. Analogously, we obtain $b,c \geq 0$. Therefore, we can define a new graph $H_1$ as follows (see Figure \ref{fig:G_H'}).
\begin{align*}
H_1=(G-B)+a \left\{v_1v_2\right\} +b \left\{v_2v_3\right\}+ c \left\{v_3v_1\right\}.
\end{align*}
By the definitions of $a,b,c$, the graph $H_1$ is $5$-regular. Moreover $H_1$ is also $5$-edge-connected. Indeed, let $Y\subseteq V(H_1)$. We can assume, without loss of generality, that $\vert Y \cap \{v_1,v_2,v_3\} \vert  \leq 1$ (otherwise, we argue by taking its complement). By the choices of $a,b$ and $c$, we have $\vert \partial_{H_1}(Y) \vert = \vert \partial_G(Y) \vert \ge 5$ and so $H_1$ is $5$-edge-connected.

Let $H'$ be the graph obtained from $G$ by identifying all vertices in $A$ to a new vertex $u$ and removing all resulting loops, see Figure \ref{fig:G_H'}. Then, $H'$ is $5$-edge-connected and every vertex is of degree $5$ except $u$. Since $|A|$ is odd,  we have that $|\partial_G(A)|$ is odd. Hence, the vertex $u$ has an odd degree of at least $5$ in $H'$. Now, by Theorem~\ref{LiftingTheo}, a new $5$-edge-connected $5$-graph $H_2$ can be obtained from $H'$ by $\frac{1}{2} (d_{H'}(u)-5)$ liftings at $u$, see Figure \ref{fig:G_H'}. We will refer to the edges of $H_2$ obtained by a lifting at $u$ as lifting edges and denote the set of all lifting edges by $\ca L$.

By the minimality of $\vert V(G) \vert$, $H_1$ has a $5$-PDPM $\{N^1_1,\ldots,N^1_5\}$ and $H_2$ has a $5$-PDPM $\{N^2_1, \ldots ,N^2_5\}$. Since $u$ has at most three neighbors in $H_2$, every perfect matching of $H_2$ contains at most one lifting edge. For each $i \in \{1, \ldots ,5\}$, let $N_i$ be the subset of edges of $H'$ defined as follows.
\begin{align*}
N_i = \begin{cases} 
 				N^2_i & \text{ if } N^2_i\cap \ca L=\emptyset ;\\
 				(N^2_i \setminus \{e\}) \cup \{e_1,e_2\} & \text{ if } N^2_i\cap \ca L=\{e\} \text{ and } e_1,e_2 \text{ are the two edges lifted to } e.
 			\end{cases}
\end{align*}

Every perfect matching of $H_1$ contains either one or three edges of $\partial_{H_1}(A)$ by Observation \ref{odd-insect}. Let $s_1$ be the number of integers $i \in \{1,\ldots,5\}$ with $\vert N^1_i \cap \partial_{H_1}(A) \vert =3$, let $s_2 = \vert \ca L \vert$ and let $s'$ be the number of integers $j \in \{1,\ldots,5\}$ with $\vert N_j \cap \partial_{H'}(u) \vert =3$. We have that $s_2=s'$.
Moreover, we have $\partial_G(A)=3s_1+(5-s_1)=5+2s_2$ and so $s_1=s_2=s'$.
Note that $\partial_{H_1}(A) = \partial_G(A)$ and recall that $H'$ is obtained from $G$ by identifying all vertices in $A$ to $u$. As a consequence, the sets of edges $N_1, \ldots, N_5$ of $H'$ and the perfect matchings $N^1_1,\ldots, N^1_5$ of $H_1$ can be combined to obtain a $5$-PDPM of $G$, a contradiction.
\end{proof}

\begin{figure}[htbp]
\centering
\subfigure[$G$]{
\begin{minipage}[t]{7cm}
\centering
\begingroup%
  \makeatletter%
  \providecommand\color[2][]{%
    \errmessage{(Inkscape) Color is used for the text in Inkscape, but the package 'color.sty' is not loaded}%
    \renewcommand\color[2][]{}%
  }%
  \providecommand\transparent[1]{%
    \errmessage{(Inkscape) Transparency is used (non-zero) for the text in Inkscape, but the package 'transparent.sty' is not loaded}%
    \renewcommand\transparent[1]{}%
  }%
  \providecommand\rotatebox[2]{#2}%
  \newcommand*\fsize{\dimexpr\f@size pt\relax}%
  \newcommand*\lineheight[1]{\fontsize{\fsize}{#1\fsize}\selectfont}%
  \ifx\svgwidth\undefined%
    \setlength{\unitlength}{144.50519396bp}%
    \ifx\svgscale\undefined%
      \relax%
    \else%
      \setlength{\unitlength}{\unitlength * \real{\svgscale}}%
    \fi%
  \else%
    \setlength{\unitlength}{\svgwidth}%
  \fi%
  \global\let\svgwidth\undefined%
  \global\let\svgscale\undefined%
  \makeatother%
  \begin{picture}(1,0.92522683)%
    \lineheight{1}%
    \setlength\tabcolsep{0pt}%
    \put(0,0){\includegraphics[width=\unitlength,page=1]{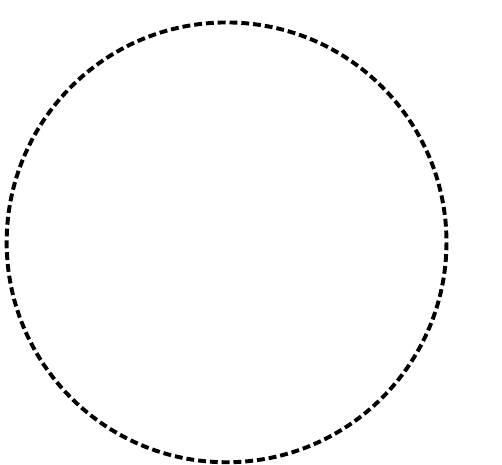}}%
    \put(1.00693368,0.7695325){\color[rgb]{0,0,0}\makebox(0,0)[rt]{\lineheight{0.3125}\smash{\begin{tabular}[t]{r}$B$\end{tabular}}}}%
    \put(0,0){\includegraphics[width=\unitlength,page=2]{FigA.pdf}}%
    \put(0.4822884,0.45832256){\color[rgb]{0,0,0}\makebox(0,0)[rt]{\lineheight{0.3125}\smash{\begin{tabular}[t]{r}$A$\end{tabular}}}}%
    \put(0,0){\includegraphics[width=\unitlength,page=3]{FigA.pdf}}%
  \end{picture}%
\endgroup%

\end{minipage}
}
\subfigure[$H_1$]{
\begin{minipage}[t]{7cm}
\centering
\begingroup%
  \makeatletter%
  \providecommand\color[2][]{%
    \errmessage{(Inkscape) Color is used for the text in Inkscape, but the package 'color.sty' is not loaded}%
    \renewcommand\color[2][]{}%
  }%
  \providecommand\transparent[1]{%
    \errmessage{(Inkscape) Transparency is used (non-zero) for the text in Inkscape, but the package 'transparent.sty' is not loaded}%
    \renewcommand\transparent[1]{}%
  }%
  \providecommand\rotatebox[2]{#2}%
  \newcommand*\fsize{\dimexpr\f@size pt\relax}%
  \newcommand*\lineheight[1]{\fontsize{\fsize}{#1\fsize}\selectfont}%
  \ifx\svgwidth\undefined%
    \setlength{\unitlength}{126.66891383bp}%
    \ifx\svgscale\undefined%
      \relax%
    \else%
      \setlength{\unitlength}{\unitlength * \real{\svgscale}}%
    \fi%
  \else%
    \setlength{\unitlength}{\svgwidth}%
  \fi%
  \global\let\svgwidth\undefined%
  \global\let\svgscale\undefined%
  \makeatother%
  \begin{picture}(1,1.00000785)%
    \lineheight{1}%
    \setlength\tabcolsep{0pt}%
    \put(0,0){\includegraphics[width=\unitlength,page=1]{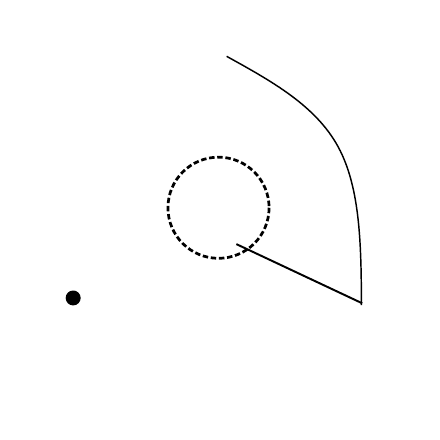}}%
    \put(0.53504477,0.51838339){\color[rgb]{0,0,0}\makebox(0,0)[rt]{\lineheight{0.31250104}\smash{\begin{tabular}[t]{r}$A$\end{tabular}}}}%
    \put(0,0){\includegraphics[width=\unitlength,page=2]{FigB.PDF}}%
  \end{picture}%
\endgroup%

\end{minipage}
}

\vspace{1cm}
\subfigure[$H'$]{
\begin{minipage}[t]{7cm}
\centering
\begingroup%
  \makeatletter%
  \providecommand\color[2][]{%
    \errmessage{(Inkscape) Color is used for the text in Inkscape, but the package 'color.sty' is not loaded}%
    \renewcommand\color[2][]{}%
  }%
  \providecommand\transparent[1]{%
    \errmessage{(Inkscape) Transparency is used (non-zero) for the text in Inkscape, but the package 'transparent.sty' is not loaded}%
    \renewcommand\transparent[1]{}%
  }%
  \providecommand\rotatebox[2]{#2}%
  \newcommand*\fsize{\dimexpr\f@size pt\relax}%
  \newcommand*\lineheight[1]{\fontsize{\fsize}{#1\fsize}\selectfont}%
  \ifx\svgwidth\undefined%
    \setlength{\unitlength}{144.45589866bp}%
    \ifx\svgscale\undefined%
      \relax%
    \else%
      \setlength{\unitlength}{\unitlength * \real{\svgscale}}%
    \fi%
  \else%
    \setlength{\unitlength}{\svgwidth}%
  \fi%
  \global\let\svgwidth\undefined%
  \global\let\svgscale\undefined%
  \makeatother%
  \begin{picture}(1,0.91882881)%
    \lineheight{1}%
    \setlength\tabcolsep{0pt}%
    \put(0,0){\includegraphics[width=\unitlength,page=1]{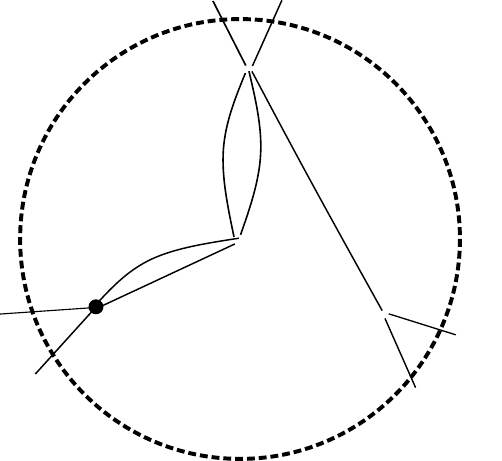}}%
    \put(0.54181646,0.45281869){\color[rgb]{0,0,0}\makebox(0,0)[rt]{\lineheight{0.3125}\smash{\begin{tabular}[t]{r}$u$\end{tabular}}}}%
    \put(1.00693605,0.7697951){\color[rgb]{0,0,0}\makebox(0,0)[rt]{\lineheight{0.3125}\smash{\begin{tabular}[t]{r}$B$\end{tabular}}}}%
    \put(0,0){\includegraphics[width=\unitlength,page=2]{FigC.pdf}}%
  \end{picture}%
\endgroup%

\end{minipage}
}
\subfigure[$H_2$]{
\begin{minipage}[t]{7cm}
\centering
\begingroup%
  \makeatletter%
  \providecommand\color[2][]{%
    \errmessage{(Inkscape) Color is used for the text in Inkscape, but the package 'color.sty' is not loaded}%
    \renewcommand\color[2][]{}%
  }%
  \providecommand\transparent[1]{%
    \errmessage{(Inkscape) Transparency is used (non-zero) for the text in Inkscape, but the package 'transparent.sty' is not loaded}%
    \renewcommand\transparent[1]{}%
  }%
  \providecommand\rotatebox[2]{#2}%
  \newcommand*\fsize{\dimexpr\f@size pt\relax}%
  \newcommand*\lineheight[1]{\fontsize{\fsize}{#1\fsize}\selectfont}%
  \ifx\svgwidth\undefined%
    \setlength{\unitlength}{144.45589866bp}%
    \ifx\svgscale\undefined%
      \relax%
    \else%
      \setlength{\unitlength}{\unitlength * \real{\svgscale}}%
    \fi%
  \else%
    \setlength{\unitlength}{\svgwidth}%
  \fi%
  \global\let\svgwidth\undefined%
  \global\let\svgscale\undefined%
  \makeatother%
  \begin{picture}(1,0.91882881)%
    \lineheight{1}%
    \setlength\tabcolsep{0pt}%
    \put(0,0){\includegraphics[width=\unitlength,page=1]{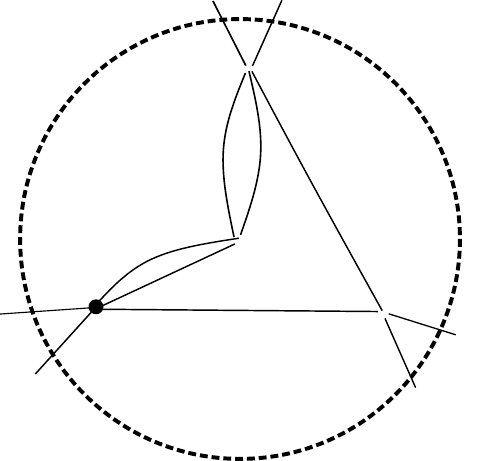}}%
    \put(0.53955223,0.45508292){\color[rgb]{0,0,0}\makebox(0,0)[rt]{\lineheight{0.3125}\smash{\begin{tabular}[t]{r}$u$\end{tabular}}}}%
    \put(1.00693605,0.7697951){\color[rgb]{0,0,0}\makebox(0,0)[rt]{\lineheight{0.3125}\smash{\begin{tabular}[t]{r}$B$\end{tabular}}}}%
    \put(0,0){\includegraphics[width=\unitlength,page=2]{FigD.pdf}}%
  \end{picture}%
\endgroup%

\end{minipage}
}
\caption{An example for the graphs $H_1$, $H'$ and $H_2$ obtained from $G$ in Theorem \ref{eq-st1}.}
\label{fig:G_H'}
\end{figure}

\section{Acknowledgments}

Major parts of this work were carried out during a stay of Davide Mattiolo at Paderborn University, sponsored by the Heinrich Hertz-Stiftung. 

\bibliography{Lit_reg_graphs}{}
\addcontentsline{toc}{section}{References}
\bibliographystyle{plain}

\newpage
\appendix

\section{A sketch of the proof of Theorem \ref{FR_implication_cubic_case}}\label{Prof_FR_implication_cubic_case}

In order to prove Theorem \ref{FR_implication_cubic_case} we adjust Theorem \ref{eq-st1} as follows.

\begin{theo}\label{eq-5reg}
 The following statements are equivalent.\\
 (i) Every  $5$-edge-connected $5$-graph with an underlying cubic graph 
 has a 2-$PDPM$. \\
 (ii) For every $5$-edge-connected $5$-graph $G$ with an underlying cubic graph and every simple $e\in E(G)$, there is a 2-$PDPM$ containing $e$. \\
 (iii) For every $5$-edge-connected $5$-graph $G$ with an underlying cubic graph and every simple $e\in E(G)$, there is a 2-$PDPM$ avoiding $e$.\\
 (iv)  For every $5$-edge-connected $5$-graph $G$ with an underlying cubic graph and every simple $e \in E(G)$ and every two parallel edges $e_1,e_2$ adjacent with $e$, there is a 2-$PDPM$ containing $e$
 and avoiding $e_1, e_2$. 
\end{theo}

\begin{proof}
Clearly,  each of $(ii)$, $(iii)$ and $(iv)$ implies 
$(i)$. Thus, it suffices to prove that $(i)$ implies $(ii)$; $(i)$ implies $(iii)$; and $(ii)$ implies $(iv)$.

	$(i) \Rightarrow (ii), (iii)$. Let $G$ be a $5$-edge-connected $5$-graph whose underlying graph is cubic and let $e=vv_1$ be a simple edge of $G$. Let $H$ and $H'$ be the graphs constructed in the part "$(i)$ $\Rightarrow$ $(ii), (iii)$" of the proof of Theorem \ref{eq-st1} by using $C_{2r}$ and 
$r$ copies of $G$ in the case $r=5$ (see Figures \ref{fig1} (a) and \ref{fig2} (a)). Clearly, the graph $H'$ can be constructed from $H$ such that every vertex of $V(H') \setminus \{u, u' \}$ has degree $3$ in the underlying graph of $H'$. Let $W=W_5 + E(C_5)$. Now, according to Definition \ref{Replace-H} replace $u$ by $(W^{1},w^{1})$ and replace $u'$ by $(W^{2},w^{2})$, where $W^l$ is a copy of $W$ and $w^l$ is the vertex of $W^l$ corresponding to $w$, for $l \in \{1,2\}$. The resulting graph, denoted by $H''$, is a $5$-edge-connected $5$-graph by Lemma \ref{GU-rcon}. Since its underlying graph is cubic, $H''$ has two disjoint perfect matchings $N_1,N_2$ by statement $(i)$. Let $N=N_1 \cup N_2$ and recall that $I=\{1,3,5,7,9\}$. For every $i \in I$ and $j \in \{1,2\}$, Observation \ref{odd-insect} implies $m_{ij} \in \{1,3\}$, where $m_{ij}=\vert \partial_{H''}(V(G^i)\setminus \{v^i\}) \cap N_j \vert$. Furthermore, we have $\vert \partial_{H''}(V(W^l)\setminus \{w^l\}) \cap N_j \vert \in \{1,3\}$ for every $l,j \in \{1,2\}$ also by Observation \ref{odd-insect}. Thus, $\vert \partial_{H''}(V(W^l)\setminus \{w^l\}) \cap N \vert \in \{2,4\}$ for every $l \in \{1,2\}$. As a consequence, there is an integer $i \in I$ such that $N$ does not contain the unique edge in $[v_1^i,V(W^1)\setminus\{w^1\}]_{H''}$. We have $m_{i1} = m_{i2} = 1$. Therefore, $G^i$ has two disjoint perfect matchings such that $v^iv_1^i$ is in none of them, which proves statement $(iii)$. For statement  $(ii)$, we consider the following cases.

	{\bf Case 1}. $\vert N \cap \partial_{H''}(V(W^1)\setminus \{w^1\}) \vert=\vert N \cap \partial_{H''}(V(W^2)\setminus \{w^2\}) \vert=2$.

	In this case, $H'$ has a $2$-PDPM, and hence, statement $(ii)$ follows by the same argumentation as in the proof of Theorem \ref{eq-st1} part "$(i) \Rightarrow (ii), (iii)$". 

	{\bf Case 2}. Without loss of generality $\vert N_1 \cap \partial_{H''}(V(W^1)\setminus \{w^1\}) \vert =3$.

	In this case, there is an integer $i \in I$, say $i=1$, such that $N_1$ contains the unique edge in $[v_1^i,V(W^1)\setminus\{w^1\}]_{H''}$ and the unique edge in $[v_1^{i+2},V(W^1)\setminus\{w^{1}\}]_{H''}$. The set $N_1$ contains exactly one edge incident with $u_2$ and thus, $m_{11} = 1$ or $m_{31} =1$ by the construction of $H''$. Therefore, $G^1$ has two disjoint perfect matchings such that $v^1v_1^1$ is in one of them or $G^3$ has two disjoint perfect matchings such that $v^3v_1^3$ is in one of them, which proves statement $(ii)$.

	{\bf Case 3}. Without loss of generality $\vert N_1 \cap \partial_{H''}(V(W^2)\setminus \{w^2\}) \vert =3$.

	In this case, there is an integer $i \in I$, say $i=3$, such that $N_1$ contains the unique edge in $[u_{i-1},V(W^2)\setminus\{w^2\}]_{H''}$ and the unique edge in $[u_{i+1},V(W^2)\setminus\{w^{2}\}]_{H''}$. As a consequence, $N_1$ contains the unique edge in $[v_1^3,V(W^1)\setminus\{w^{1}\}]_{H''}$ and $m_{31}=1$. Therefore, $G^3$ has two disjoint perfect matchings such that $v^3v_1^3$ is in one of them, which proves statement $(ii)$.

	$(ii)$ $\Rightarrow$ $(iv)$. 
Let $G$ be a $5$-edge-connected $5$-graph whose underlying graph is cubic, let $v \in V(G)$ and let $N_G(v)=\{v_1,v_2,v_3\}$. Furthermore, let $\mu_G(v,v_1)=1$, $\mu_G(v,v_2)=\mu_G(v,v_3)=2$, let $e$ be the edge connecting $v$ and $v_1$ and let $e_1,e_2$ be the two parallel edges connecting $v$ and $v_2$. We show that there are two disjoint perfect matchings such that their union contains $e$ but neither $e_1$ nor $e_2$.

	Let $G^1$ and $G^3$ be two copies of $G$ in which the vertices and edges are labeled accordingly by using an upper index. Let $H$ be the graph constructed in the part "$(ii)$ $\Rightarrow$ $(iv)$" of the proof of Theorem \ref{eq-st1} by using $K_4$ in the case $r=5$, see Figure \ref{fig3} (a). According to Definition \ref{Replace-H}, construct a new graph $H'$ from $H$ by replacing $u_1$ with $(G^1,v^1)$ and replacing $u_3$ with $(G^3,v^3)$ such that $\mu_{H'}(v_1^1,v_1^3)=1$ and $\mu_{H'}(v_2^1,u_2)=\mu_{H'}(v_2^3,u_2)=\mu_{H'}(v_3^1,u_4)=\mu_{H'}(v_3^3,u_4)=2$. The graph $H'$ is $5$-edge-connected and $5$-regular by Lemma \ref{GU-rcon} and its underlying graph is cubic. Therefore, by statement $(ii)$ there are two disjoint perfect matchings $N_1,N_2$ of $H'$ such that $u_2u_4 \in N_1$. By Observation \ref{odd-insect}, we have $v_1^1v_1^3 \in N_1$ and $\vert \partial_{H'}(V(G^{i}) \setminus \{v^i\}) \cap N_j \vert =1$ for every $i \in \{1,3\}$ and every $j \in \{1,2\}$. Furthermore, $N_1 \cup N_2$ either does not contain the two edges connecting $v_2^1$ and $u_2$ or does not contain the two edges connecting $v_2^3$ and $u_2$. In the first case, $G^1$ has two disjoint perfect matchings such that their union contains $e^1$ but neither $e_1^1$ nor $e_2^1$; in the second case $G^3$ has two disjoint perfect matchings such that their union contains $e^3$ but neither $e_1^3$ nor $e_2^3$. This proves statement $(iv)$.
\end{proof}

Theorem \ref{FR_implication_cubic_case} can be proved like Theorem \ref{FR_implication} by using Theorem \ref{eq-5reg} instead of Theorem \ref{eq-st1}.

\end{document}